\documentclass[11pt]{amsart}

\usepackage{amsmath, amsthm, graphics,setspace,xypic}
\usepackage{latexsym}
\usepackage{mathtools}
\usepackage{eucal}
\usepackage{caption}

\usepackage{amssymb}
\usepackage{color}
\usepackage{amscd}
\usepackage{palatino}
\usepackage{latexsym}
\usepackage{epsfig}
\usepackage{graphicx}
\usepackage{amsfonts}
\usepackage{psfrag}
\usepackage[all]{xy}
\usepackage{youngtab, ytableau}

\usepackage{url}
\usepackage{cite}

\usepackage[margin=1in]{geometry}

\input xy
\xyoption{all}
\UseComputerModernTips

\numberwithin{equation}{section}
\numberwithin{figure}{section}

\newtheorem{theorem}{Theorem}[section]
\newtheorem{lemma}[theorem]{Lemma}
\newtheorem{proposition}[theorem]{Proposition}

\newtheorem{remark}[theorem]{Remark}
\newtheorem{example}[theorem]{Example}

\theoremstyle{definition}
\newtheorem{definition}[theorem]{Definition}

\newcommand{\C}{{\mathbb{C}}}

\newcommand{\Z}{{\mathbb{Z}}}

\newcommand{\B}{\mathcal{B}}

\newcommand{\fg}{\mathfrak{g}}

\DeclareMathOperator{\Par}{Par}

\DeclareMathOperator{\Span}{span}

\DeclareMathOperator{\im}{im}

\DeclareMathOperator{\Flags}{\mathcal{F}lags}

\usepackage{multicol}
\usepackage{color} 

\definecolor{gold}{rgb}{0.85,.66,0}
\definecolor{cherry}{rgb}{0.9,.1,.2}
\definecolor{burgundy}{rgb}{0.8,.2,.2}
\definecolor{orangered}{rgb}{0.85,.3,0}
\definecolor{orange}{rgb}{0.85,.4,0}
\definecolor{olive}{rgb}{.45,.4,0}
\definecolor{lime}{rgb}{.6,.9,0}
\definecolor{green}{rgb}{.2,.7,0}
\definecolor{grey}{rgb}{.4,.4,.2}
\definecolor{brown}{rgb}{.4,.3,.1}

\DeclareMathOperator{\inv}{inv}

\DeclareMathOperator{\Hess}{\mathcal{H}ess}
\DeclareMathOperator{\RS}{RS}

\DeclareMathOperator{\std}{std}

\begin{document}

\title[Hessenberg varieties associated to ad-nilpotent ideals]{Hessenberg varieties associated to ad-nilpotent ideals}

\author{Caleb Ji}
\address{Department of Mathematics\\ Columbia University in the City of New York \\ U.S.A. }
\email{caleb.ji@columbia.edu}

\author{Martha Precup}
\address{Department of Mathematics and Statistics\\ Washington University in St. Louis \\ One Brookings Drive \\ St. Louis, Missouri  63130 \\ U.S.A. }
\email{martha.precup@wustl.edu}

%


\begin{abstract} We consider Hessenberg varieties in the flag variety of $GL_n(\C)$ with the property that the corresponding Hessenberg function defines an ad-nilpotent ideal.  Each such Hessenberg variety is contained in a Springer fiber. We extend a theorem of Tymoczko to this setting, showing that these varieties have an affine paving obtained by intersecting with Schubert cells.  Our method of proof constructs an an affine paving for each Springer fiber that restricts to an affine paving of the Hessenberg variety.  We use the combinatorial properties of this paving to prove that Hessenberg varieties of this kind are connected.
\end{abstract}

\maketitle

\section{Introduction}
This paper studies topological and combinatorial properties of a certain class of Hessenberg varieties.  Hessenberg varieties, as introduced in \cite{DPS1992}, are subvarieties of the flag variety.  They are important examples of varieties whose geometry and topology can be characterized using combinatorial techniques (see, for example,  \cite{AHHM2014, Harada-Tymoczko2017,Harada-Precup2017}).  The Hessenberg variety $\Hess(X,h)$ is parametrized by two pieces of data: a matrix $X\in \mathfrak{gl}_n(\C)$ and a non-decreasing function $h:[n]\rightarrow [n]$, known as a Hessenberg function.  

Most of the existing literature on Hessenberg varieties considers only Hessenberg functions with the property that $h(i)\geq i$ for all $i$.  Tymoczko has shown that the Hessenberg varieties corresponding to such Hessenberg functions have a paving by affines \cite{Tymoczko2006}.   
This paper investigates Hessenberg varieties corresponding to Hessenberg functions such that $h(i)<i$ for all $i$.  In this case, the Hessenberg space of the function $h$  is an ad-nilpotent ideal and $\Hess(X,h)$ is a subvariety of the Springer fiber for $X$.  We construct an affine paving for these varieties and explore additional geometric and combinatorial properties. 

The fact that Hessenberg varieties of this kind are paved by affines is not new; Fresse proves this statement for a more general class of Hessenberg varieties in \cite{Fresse2016}.  While the arguments used in that paper are broader in scope, they do not compute the dimension of each affine cell in the paving. Our methods are constructive and we obtain combinatorial formulas for the dimension of the cells, recovering Tymoczko's results in this setting.  In Section~\ref{sec.paving} below, we define explicit coordinates for an affine paving of the Springer fiber. We then obtain a paving of the Hessenberg variety $\Hess(X,h)$ by setting certain coordinates equal to zero; this is recorded in Theorem~\ref{thm: main thm2}.   Our arguments are of a similar flavor as those given by Spaltenstein in~\cite{Spaltenstein1976}.  

We give two applications of our results in Section~\ref{sec.applications}.  Recall that the irreducible components of the Springer fibers are in bijection with standard tableaux.  This is one of the key conclusions of Springer theory.  The Hessenberg varieties we consider here may not be equidimensional, so the cells in the affine paving of maximal dimension are not in bijection with irreducible components.  However, Theorem~\ref{thm.maximalcells} below shows that these cells are still indexed by standard tableaux.  The second main result of Section~\ref{sec.applications}, namely Theorem~\ref{thm.connected}, proves that the Hessenberg varieties we consider are always connected (in the type $A$ case).  Example~\ref{ex.TypeC} shows that this property may not be true for analogous Hessenberg varieties defined using other classical groups. 

The constructions in this paper are motivated by the goal of better understanding the geometry of the affine paving.   Determining the closure relations between cells in the paving and identifying singularities of the irreducible components of $\Hess(X,h)$ are both interesting open questions.  Even in the case of the Springer fiber, the answer to these questions is unknown, although progress has been made in special cases \cite{Fresse2010,Fung2003,Graham-Zierau2011}.  Since the Hessenberg varieties considered here are all subvarieties of a Springer fiber, a thorough study of their geometry has the potential to shed new light on these subjects. 

The organization of this paper is as follows.  In Section 2, we review necessary definitions and prior results.  In Section~\ref{sec.dimpairs} we study the notion of Hessenberg inversions, originally introduced by Tymoczko in \cite{Tymoczko2006}.   We define certain subgroups of matrices crucial to the construction of our paving in Section~\ref{sec.Bmatrices}. Our affine paving is defined in Section~\ref{sec.paving} and we prove our main result, which is Theorem~\ref{thm: main thm2}.  Finally, we explore some combinatorial properties of our construction in Section~\ref{sec.applications}.

\vspace{.1in}

\noindent \textbf{Acknowledgements.} The authors are thankful to the anonymous referee for their feedback, which improved the exposition of Section~\ref{sec.paving} below.  The second author is supported in part by NSF DMS-1954001.


\section{Preliminaries}\label{sec.preliminaries}
Let $n$ be a positive integer and $[n]$ denote the set of positive integers $\{1,2,\ldots, n\}$. We work in type A throughout (except for Example~\ref{ex.TypeC} in Section~\ref{sec.applications}), where $GL_n(\C)$ is the group of invertible $n\times n$ complex matrices and $\mathfrak{gl}_n(\C)$ is the Lie algebra of all $n\times n$ complex matrices.  Let $B$ be the Borel subgroup of $GL_n(\C)$ consisting of upper triangular matrices and $U$ be the subgroup of upper triangular matrices with diagonal entries equal to $1$. 

The Weyl group of $GL_n(\C)$ is $S_n$, which we identify with the subgroup of permutation matrices in $GL_n(\C)$.  Given $w\in S_n$, let 
\[
\inv(w):= \{(i,j)\mid i>j \textup{ and } w(i)<w(j) \}
\]
denote the set of inversions of $w$.  Note that we adopt the nonstandard notation of listing the larger number in the pair $(i,j)\in \inv(w)$ first; this simplifies our exposition below.  The Bruhat length of a permutation $w\in S_n$ is $\ell(w):= |\inv(w)|$.


\subsection{Hessenberg Varieties} The \textbf{flag variety} is the collection of all full flags in $\C^n$,
\[
\Flags(\C^n) := 
\{ V_{\bullet}=(\{0\}\subseteq V_1 \subseteq V_2\subseteq  \ldots\subseteq \C^n)  \mid  \dim_{\C}(V_i) = i \, \textup{ for all }\, i\in[n] \}. 
\]
Given a full flag $V_\bullet$, let $\{\mathbf{v}_1, \mathbf{v}_2, \ldots, \mathbf{v}_n\}$ be a basis of $\C^n$ such that  for each $i$, $\{\mathbf{v}_1, \mathbf{v}_2, \ldots, \mathbf{v}_i\}$ is a basis for $V_i$.  We denote the flag $V_\bullet$ by $V_\bullet = (\mathbf{v}_1\mid \mathbf{v}_2\mid \cdots \mid \mathbf{v}_n)$.  Let $\{\mathbf{e}_1, \mathbf{e}_2, \ldots, \mathbf{e}_n\}$ be the standard basis of $\C^n$.  The \textbf{standard flag} $E_\bullet$ is the full flag $E_\bullet = (\mathbf{e}_1\mid \mathbf{e}_2 \mid \cdots \mid \mathbf{e}_n)$.  Every flag $V_\bullet$ is of the form $gE_\bullet$ where $g\in GL_n(\C)$ such that $g\mathbf{e}_k=\mathbf{v_k}$ and $gE_\bullet:= (g\mathbf{e}_1 \mid g\mathbf{e}_2 \mid \cdots \mid g\mathbf{e}_n)$.

\begin{remark} The flag variety identifies with the homogeneous space $GL_n(\C)/B$ via the map $gB \mapsto gE_\bullet$.  In this paper, we interchange notation for the flag $gE_\bullet$ and coset $gB$ whenever it is convenient.
\end{remark}

A Hessenberg variety in $\Flags(\C^n)$ is specified by two pieces of data: a Hessenberg function and an element of $\mathfrak{gl}_n(\C)$.   A \textbf{Hessenberg function} is a function $h: [n] \to [n]$ such that $h(i)\leq h(i+1)$ for all $i \in [n-1]$. We frequently write a Hessenberg function by listing its values in sequence, i.e., $h = (h(1), h(2), \ldots, h(n))$. We now define the main objects of interest in this paper.

\begin{definition} Let $h: [n]\to [n]$ be a Hessenberg function and $X\in \mathfrak{gl}_n(\C)$.  The \textbf{Hessenberg variety} associated to $h$ and $X$ is 
\[
\Hess(X,h) := \{V_{\bullet}\in \Flags(\C^n) \mid X(V_i)\subseteq V_{h(i)} \, \textup{ for all }\, i\in [n] \}.
\]
\end{definition}

If $V_\bullet=(\mathbf{v}_1\mid \mathbf{v}_2\mid \cdots \mid \mathbf{v}_n)$ then $V_\bullet\in \Hess(X,h)$ if and only if $X\mathbf{v}_i\in \Span\{\mathbf{v}_1, \ldots, \mathbf{v}_{h(i)}\}$ for all $i\in [n]$.  When $X\in \fg$ is a nilpotent matrix and $h=(0,1,\ldots, n-2, n-1)$, the variety $\B(X, h)$ is the \textbf{Springer fiber} of $X$, which we denote by $\B^X$.

The following remark indicates that we may choose any matrix within a given conjugacy class for our computations without alternating the geometric invariants of the corresponding Hessenberg variety.

\begin{remark}\label{rem.conjugation} Given a fixed Hessenberg function $h$, we have $\Hess(X,h)\simeq \Hess(g^{-1}Xg, h)$ for all $g\in GL_n(\C)$. 
\end{remark}

Most of the existing literature on Hessenberg varieties assumes that the Hessenberg function also satisfies the condition that $h(i)\geq i$.  The main reason is that this condition on the Hessenberg function ensures $\Hess(X,h)\neq \emptyset$ for all $X\in \mathfrak{gl}_n(\C)$. One of the main purposes of this paper is to explore Hessenberg varieties corresponding to Hessenberg functions with the property that $h(i)<i$ for all $i$. This is exactly the case in which the corresponding \textbf{Hessenberg space}, defined by:
\begin{eqnarray*}
H(h): = \Span\{E_{ij}\mid i\leq h(j) \}\subseteq \mathfrak{gl}_n(\C),
\end{eqnarray*}
is an ad-nilpotent ideal (that is, its lower central series is finite).  Thus, \textbf{for the remainder of this manuscript we assume that any Hessenberg function $h:[n]\to [n]$ satisfies $h(i)<i$ for all $i\in [n]$.}

Let $\mathcal{H} : = \{h: [n]\to[n] \mid  h(i+1)\geq h(i) \textup{ and } h(i)<i \}$ denote the set of all Hessenberg functions satisfying the condition that $h(i)<i$.  There is a partial ordering on this set defined by 
\[
h_1 \preceq h_2 \Leftrightarrow  h_1(i)\leq h_2(i) \textup{ for all $i$}
\]
for $h_1,h_2\in \mathcal{H}$.  A partial order like this one is studied by Drellich in \cite{Drellich2017}.  It follows directly from the definition that if $h_1\preceq h_2$ then $\Hess(X,h_1)\subseteq \Hess(X,h_2)$ for all $X\in \mathfrak{gl}_n(\C)$.  Note that our set of Hessenberg functions contains a unique maximal element with respect to $\preceq$, namely the Hessenberg function $h=(0,1, \ldots, n-1)$. Thus $\Hess(X,h) \subseteq \B^X$ for any nilpotent matrix $X\in \mathfrak{gl}_n(\C)$ and $h\in \mathcal{H}$.

\subsection{Affine Pavings}\label{subsec: affine paving}  The first main goal of this manuscript is to demonstrate an affine paving of the Hessenberg variety $\Hess(X,h)$ obtained by intersecting with the Schubert cells.  We do so by first constructing an explicit affine paving of the Springer fiber $\B^X$.  We then prove that this paving restricts to a paving of $\Hess(X,h)$ in a natural way.   Note that it is very well known that Springer fibers are paved by affines \cite{Spaltenstein1976,Fresse2010}, and Tymoczko's results prove that such a paving can be obtained by intersecting with the Schubert cells \cite{Tymoczko2006}, so our result in that case is not new.

\begin{definition}  A \textbf{paving} of an algebraic variety $Y$ is a filtration by closed subvarieties
\[
Y_0\subset Y_1 \subset \cdots \subset Y_i \subset \cdots \subset Y_d=Y.
\]	
A paving is \textbf{affine} if every $Y_i-Y_{i-1}$ is isomorphic to a finite disjoint union of affine spaces; we calls these spaces the \textbf{affine cells} of the paving.  
\end{definition}

An affine paving allows us to compute the {Betti numbers} of an algebraic variety $Y$, as shown in~\cite[1.9.1, 19.1.11]{Fulton}.  In the statement below, $H_c^*(Y)$ denotes cohomology with compact support of the algebraic variety $Y$. 

\begin{lemma}\label{betti}  Let $Y$ be an algebraic variety with an affine paving, $Y_0\subset Y_1 \subset \cdots \subset Y_i \subset \cdots \subset Y_d=Y$.  Then the nonzero cohomology groups of $Y$ are given by $H_c^{2k}(Y)= \Z^{n_k}$ where $n_k$ denotes the number of affine components of dimension $k$. 
\end{lemma}

In this paper we apply the lemma for $Y$ a complex projective variety, so $H^*(Y)=H_c^*(Y)$.
There is a well known affine paving of $\Flags(\C^n)$ induced by the \textbf{Bruhat decomposition}: 
\[
\Flags(\C^n) = \bigsqcup_{w\in S_n} C_w \textup{ where } C_w = BwE_\bullet.
\] 
The $B$-orbit $C_w$ is called the \textbf{Schubert cell} indexed by $w\in S_n$.   It is well known that each Schubert cell is isomorphic to the subgroup $U^w:= U\cap wU^-w^{-1}$ where $U^-$ is the subgroup of lower triangular matrices with diagonal entries equal to $1$. In other words, each flag $bwE_\bullet\in C_w$ can be written uniquely as $bwE_\bullet = uwE_\bullet$ for some $u\in U^w$.  Since $U^w$ is a unipotent subgroup we have $U^w\simeq Lie(U^w)$, where $Lie(U^w)$ an affine space of dimension $\ell(w)$.  Therefore $C_w\simeq \C^{\ell(w)}$.  It is well known that $\overline{C_w} = \bigcup_{y\leq w}C_y$, where $\leq$ denotes the Bruhat order on $S_n$.  Thus, the Schubert cells are affine cells for the paving of $\Flags(\C^n)$ defined by
\[
\B_0\subset \B_1 \subset \cdots \subset \B_{\frac{n(n-1)}{2}}= \Flags(\C^n)\, \textup{ where } \, \B_i:=\bigcup_{\substack{w\in S_n\\ \ell(w)=i}} \overline{C_w}.
\]  
We prove $\Hess(X,h)$ has an affine paving by considering the intersections $C_w\cap \Hess(X,h)$.

\begin{remark}\label{rem.paveit} It follows from the discussion in \cite[\S 2.2]{Precup2013} that in order to prove $\Hess(X,h)$ is paved by affines, it suffices to prove $\Hess(X,h)\cap C_w$ is isomorphic to affine space $\C^d$ with $d\in \Z_{\geq0}$. 
\end{remark}

\subsection{Factorization}\label{subsec: factorization} 

We now describe a method for identifying a portion of any Schubert cell in $\Flags(\C^n)$ with a Schubert cell in the flag variety associated to $GL_{n-1}(\C)$, namely $\Flags(\C^{n-1})$. We view $GL_{n-1}(\C)$ as a subgroup of $GL_n(\C)$ by identifying it with its image under the map:
\begin{eqnarray}\label{eqn.induction}
GL_{n-1}(\C) \to GL_n(\C) ; \; a\mapsto \left[\begin{array}{c|c} a & 0 \\\hline 0 & 1 \end{array}\right] \textup{ for all } a\in GL_{n-1}(\C).
\end{eqnarray}
Let $U_0$ be the unipotent subgroup of $GL_{n-1}(\C)$ of upper triangular matrices with diagonal entries equal to $1$. We view $U_0$ as a subgroup of $GL_n(\C)$ via $U_0 = \{u\in U \mid u_{in}=0 \textup{ if } i\neq n\}$.  Similarly, we identify $S_{n-1}$ with the subgroup $\{\sigma\in S_n\mid \sigma(n)=n\}$ of $S_n$. 

Each permutation $w\in S_n$ can be factorized uniquely as 
\begin{eqnarray}\label{eqn: factorization}
vy   \textup{ where }v=s_{i}s_{i+1}\cdots s_{n-2}s_{n-1},\textup{ for } i=w(n) \textup{ and }  y\in S_{n-1}.
\end{eqnarray}

Here $s_j$ denotes the simple transposition swapping $j$ and $j+1$. Note that $v$ in the factorization above is called the shortest left coset representative for $w=vy$ with respect to the Young subgroup $S_{n-1} = \left< s_1, \ldots, s_{n-2} \right>$, see \cite[Proposition 2.4.4]{Bjorner-Brenti}.  In one-line notation, we have that $v$ is the permutation with the property that $v(n)=i$ and all remaining values are placed in positions $1,2,\ldots, n-1$ of the one-line notation for $v$ in increasing order; $y$ is the unique permutation with the property that $y(n)=n$ and the rest of the entries in the one-line notation for $y$ are in the same relative order as the entries of $w$.   

 The factorization given in~\eqref{eqn: factorization} satisfies the condition that $\ell(w) = \ell(v)+\ell(y)$ and:
\begin{eqnarray}\label{eqn: inversions}
\inv(w) =\inv(y) \sqcup  y^{-1}(\inv(v)).
\end{eqnarray}

\begin{example}\label{ex: inversions} Let $w=[3,4,1,2]=s_2s_3s_1s_2\in S_4$. Then $w(4)=2$ and we see that 
\[
w= v y \,\textup{ where }\, v=s_2s_3 \, \textup{ and }\, y= s_1s_2.
\]
In one-line notation, $v=[1,3,4,2]$ and $y=[2,3,1, 4]$.  We have that 
\[
\inv(w) = \{(3,2), (3,1) , (4,2), (4,1)\} 
\]
where $\inv(y) = \{(3,2), (3,1) \}$ and $y^{-1}(\inv(v)) = \{ (4,2), (4,1) \}$, confirming~\eqref{eqn: inversions}.
\end{example}

Recall that $U^w:=U\cap wU^-w^{-1}$. In the special case where $v=s_is_{i+1}\cdots s_{n-2}s_{n-1}$ for some $i\in [n]$ we have $U^{v} = \{u\in U \mid u_{kj}=0 \, \textup{ for all }\, k\neq i,\, k<j\}$, i.e., $U^{v}$ is the $i$-th row of $U$, which we denote by $U_i$.    The next lemma tells us there is a factorization of the elements of $U^w$ that is compatible with the factorization of permutations given in~\eqref{eqn: factorization} above.  This is a special case of \cite[Proposition 28.1]{HumphreysLAG}.

\begin{lemma}\label{lemma: Schubert cell decomp} Suppose $w\in S_n$ and let $w=v y$ be the factorization given in~\eqref{eqn: factorization} with $i=w(n)=v(n)$. For each $u\in U^w$ the product $uw$ can be written uniquely as 
\[
uw = u_i v u_0 y \;\textup{ for some }\; u_i\in U_i=U^v \textup{ and } u_0\in U^y.
\]
\end{lemma}

Lemma~\ref{lemma: Schubert cell decomp} gives us an inductive decomposition of each Schubert cell, as we now explain. 
Let $w\in S_n$ and $w=vy$ be the factorization from~\eqref{eqn: factorization}.  Given $uwE_\bullet$ with $u\in U$, by Lemma~\ref{lemma: Schubert cell decomp} we may write $uwE_\bullet = u_ivu_0yE_\bullet$ with $u_i\in U_i$ and $u_0\in U_0$. (Note that since $y(n)=n$ we have $U^y \subseteq U_0$.) We obtain an isomorphism,
\begin{eqnarray}\label{eqn: induction}
\rho: C_w \to \C^{\ell(v)} \times C_y,\; uwE_\bullet=u_ivu_0yE_{\bullet} \mapsto (u_i, u_0yE_{\bullet}')
\end{eqnarray}
where $E_\bullet':=(\mathbf{e}_1\mid \mathbf{e}_2\mid \ldots\mid \mathbf{e}_{n-1})$ is the standard flag in $\C^{n-1} = \Span\{\mathbf{e}_1, \mathbf{e}_2, \ldots, \mathbf{e}_{n-1}\}$ and we identify $U_i$ with affine space via $u_i\mapsto (u_{ij})_{j>i}$, where $u_i = I_n + \sum_{j>i} u_{ij}E_{ij}$.  We make the identification $U_i\simeq \C^{\ell(v)}$ implicitly throughout this paper.   Let 
\begin{eqnarray}\label{eqn.proj-morphism}
\rho_1: C_w \to \C^{\ell(v)} ,\; uwE_\bullet = u_ivu_0E_\bullet \mapsto u_i
\end{eqnarray}
be the map obtained from $\rho$ via composition with projection to the first factor.  Denote by $P$ the maximal standard parabolic subgroup with Levi subgroup equal to the image of~\eqref{eqn.induction}.  That is, $P=\{g\in G\mid g_{nj}=0 \textup{ for all $j<n$ }\}$.  The map $\rho_1$ can be identified with the restriction to $C_w$ of the canonical projection $G/B \to G/P$; in the notation of cosets we have $\rho_1(uwB) = uwP = u_ivP$.  Thus $\rho_1$ is a morphism of varieties and commutes with the action of $B$.


\section{Hessenberg Inversions}\label{sec.dimpairs}


Let $\lambda=(\lambda_1, \lambda_2, \ldots, \lambda_k)$ be a weak composition of $n$ and $\Par(\lambda)$ be the partition we obtain from $\lambda$ by rearranging the parts of $\lambda$ in decreasing order.  We begin by fixing an element $X_{\lambda}$ in the conjugacy class $\mathcal{O}_{\Par(\lambda)}$ of nilpotent matrices of Jordan type $\lambda$.

\begin{definition}\label{def: fixed element} Let $\lambda=(\lambda_1, \lambda_2, \ldots, \lambda_k)$ be a weak composition of $n$ drawn as a diagram, namely with $k$ rows of boxes so that the $i^{th}$ row from the top has $\lambda_i$ boxes.  The \textbf{base filling} of $\lambda$ is obtained as follows.  Fill the boxes of $\lambda$ with integers $1$ to $n$ starting at the bottom of the leftmost column and moving up the column by increments of one.  Then move to the lowest box of the next column and so on.   Denote the base filling of $\lambda$ by $R(e)$. We now define:
\begin{eqnarray}\label{eqn: X-definition}
X_\lambda := \sum_{(\ell,r)} E_{\ell r}
\end{eqnarray}
where the sum is taken over the set of all pairs $(\ell,r)$ such that $r$ labels the box directly to the right of $\ell$ in the base filling of $\lambda$.
\end{definition}

\begin{example} If $n=7$ and $\lambda=(2,3,1,1)$ then the base filling of $\lambda$ is:
\[\ytableausetup{centertableaux}
\begin{ytableau}
4 & 6 \\
3 & 5 & 7 \\
2 \\
1\\ 
\end{ytableau} 
\quad\; \textup{ and we have } \; 
X_{(2,3,1,1)} = \begin{bmatrix} 0 & 0 & 0 & 0 & 0 & 0 & 0\\ 0 & 0 & 0 & 0 & 0 & 0 & 0 \\ 0 & 0 & 0& 0 &1 & 0 & 0\\ 0 & 0 & 0 & 0 & 0 & 1 & 0\\ 0 & 0 & 0 & 0 & 0 & 0 & 1\\ 0 & 0 & 0 & 0 & 0 & 0 & 0\\ 0 & 0 & 0 & 0 & 0 & 0 & 0  \end{bmatrix}.
\]
\end{example}

For each $w\in S_n$, let $R(w)$ denote the tableau of composition shape $\lambda$ obtained by labeling the $i$-th box in the base filling of $\lambda$ by $w^{-1}(i)$.   We say that $R(w)$ is \textbf{$h$-strict} if  $\ell \leq h(r)$ whenever $\ell$ labels a box directly to the left of $r$ in $R(w)$. Let $\RS_h(\lambda)$ denote the set of all $h$-strict tableaux of composition shape $\lambda$. The set of $h$-strict tableaux determines which Schubert cells intersect the Hessenberg variety.   This is proved by Tymoczko in \cite[Theorem 7.1]{Tymoczko2006} for Hessenberg varieties associated to Hessenberg functions such that $h(i)\geq i$ for all $i$.  The proof below is the same; we give a sketch using our notation for the reader's convenience.

\begin{lemma}\label{lemma: h-strict tab} Let $w\in S_n$ and $h\in \mathcal{H}$.  Then $C_w\cap \Hess(X_\lambda, h)\neq \emptyset$ if and only if $R(w)\in \RS_h(\lambda)$.
\end{lemma}
\begin{proof}[Sketch of proof] By definition, $w E_\bullet \in \Hess(X_\lambda, h)$ if and only if 
\[
X_\lambda \mathbf{e}_{w(r)} \in \Span\{\mathbf{e}_{w(1)}, \ldots, \mathbf{e}_{w(h(r))}\}
\] 
for all $r\in [n]$.  Suppose $\ell$ labels the box directly to the left of $r$ in $R(w)$.  The $w(\ell)$ labels the box directly to the left of $w(r)$ in $R(e)$.  Since $X_\lambda\mathbf{e}_{w(r)} = \mathbf{e}_{w(\ell)}$ we therefore have $w E_\bullet \in \Hess(X_\lambda, h)$ if and only if $\ell \leq h(r)$ for any such pair $(\ell, r)$.

To complete the proof we have only to show that $\Hess(X_\lambda, h)\cap C_w\neq \emptyset$ implies $wE_\bullet\in \Hess(X_\lambda, h)$.  Assume $uwE_\bullet \in \Hess(X_\lambda, h)$ for some $u\in U^w$.  Then
\[
X_\lambda u\mathbf{e}_{w(r)} \in \Span \{u\mathbf{e}_{w(1)}, \ldots, u\mathbf{e}_{w(h(r))}\} \Leftrightarrow (u^{-1}X_\lambda u) \mathbf{e}_{w(r)} \in \Span\{\mathbf{e}_{w(1)}, \ldots, \mathbf{e}_{w(h(r))}\}
\]  
for all $r\in [n]$.  The desired statement now follows immediately from the fact that the pivots of $u^{-1}X_\lambda u$ are in the same position as the pivots of $X_\lambda$ (as proved by Tymoczko in \cite[Proposition 4.6]{Tymoczko2006}).  
\end{proof}

When $h=(0,1,\ldots, n-1)$ we have that $\RS_h(\lambda)=:\RS(\lambda)$ is the set of tableaux of composition shape $\lambda$ which are row-strict, that is, increasing across rows.  By definition, $\RS_h(\lambda)\subseteq \RS(\lambda)$ for all Hessenberg functions $h\in \mathcal{H}$. Our next definition comes from \cite[Theorem 7.1]{Tymoczko2006}, see \cite{Precup-Tymoczko2019} also.

\begin{definition}\label{defn: Hess inversions} Let $\lambda$ be a weak composition of $n$ and $k,\ell\in [n]$.   We say $(k,\ell)$ is a \textbf{Hessenberg inversion} of $R(w)$ for $w\in S_n$ if $k>\ell$ and:
\begin{enumerate}
\item $k$ occurs in a box below $\ell$ and in the same column or in any column strictly to the left of the column containing $\ell$ in $R(w)$, and
\item if the box directly to the right of $\ell$ in $R(w)$ is labeled by $r$, then $k\leq h(r)$.
\end{enumerate}
Denote the set of Hessenberg inversions in $R(w)$ by $\inv_{h,\lambda}(w)$.
\end{definition}

Note that if the pair $(k,\ell)$ satisfies condition (1), then $(k, \ell)\in \inv(w)$; so Hessenberg inversions are a subset of the inversions of $w$.

\begin{remark}\label{rem: k-inversions} If $(k, \ell_1), (k,\ell_2)\in \inv_{h,\lambda}(w)$ then $k$, $\ell_1$, and $\ell_2$ are all in different rows of $R(w)$, or equivalently, $w(k), $$w(\ell_1)$, and $w(\ell_2)$ are all in different rows of $R(e)$.  Indeed if $\ell_1$ fills a box to the left of $\ell_2$ and in the same row, the assumption that $(k,\ell_1)$ is a Hessenberg inversion implies that $k$ is less than every entry to the right of $\ell_1$, implying $(k,\ell_2)$ cannot be an inversion.  

Now suppose $\ell_1$ and $k$ occur in the same row. Since $(k,\ell_1)$ is a Hessenberg inversion, $k$ must occur to the left of $\ell_1$.  On the other hand, $R(w)$ must be row-strict and $k>\ell_1$ so we obtain a contradiction.
\end{remark}

\begin{example} Let $n=7$ and $\lambda=(2,3,1,1)$.  The tableau $R(w)$ for $w=[4,3,1,6,5,7,2]$ is shown below. 
\[\ytableausetup{centertableaux}
\begin{ytableau}
1 & 4 \\
2 & 5 & 6 \\
7 \\
3\\ 
\end{ytableau} 
\]
If $h=(0,1,2,3,4,5,6)$ then $R(w)$ has inversion set $\inv_{\lambda, h}(w)=\{ (7,6), (7,4), (5,4), (3,2), (3,1), (2,1) \}$.  Note that $(7,5)\in \inv(w)$ but $(7,5)$ is not a Hessenberg inversion since $7 \nleq 5=h(6)$.  If $h=(0,0,1,2,3,3,3)$ then the inversion set becomes $\inv_{\lambda,h}(w) = \{ (7,6), (7,4), (5,4), (3,2), (2,1)\}$ since $3\nleq 2=h(4)$ now.
\end{example}

When $h=(0,1,2,\ldots, n-1)$, we called the pairs in Definition~\ref{defn: Hess inversions} above \textbf{Springer inversions}, denoted $\inv_\lambda(w)$ in this case.  We let 
\[
\inv_{\lambda}^k(w) := \{(k,\ell) \mid 1\leq \ell <k \textup{ and } (k,\ell)\in \inv_\lambda(w)\}
\]
so $\inv_\lambda(w) = \bigsqcup_{k=2}^n \inv_{\lambda}^k(w)$. We set $d_k:= |\inv_\lambda^k(w)|$ for all $2\leq k \leq n$.

Let $w\in S_n$ such that $R(w)\in \RS(\lambda)$.  Since $R(w)$ is row-strict, the box labeled by $n$ must appear at the end of a row.  Let $\lambda'$ be the composition of $n-1$ we obtain from $\lambda$ by deleting the box labeled by $n$ in $R(w)$, or equivalently, deleting the box labeled by $i=w(n)$ in $R(e)$. Our next lemma shows that the Hessenberg inversions of $w$ are well-behaved with respect to the decomposition of $\inv(w)$ given in~\eqref{eqn: inversions}.

\begin{lemma}\label{lemma: Springer case induction} Suppose $R(w)\in \RS(\lambda)$ and $w=vy$ is the factorization from~\eqref{eqn: factorization} with $i=w(n)$. Then $R(y)\in \RS(\lambda')$ and 
\[
\inv_{\lambda'}(y) = \inv_{\lambda}(w) \setminus \inv_{\lambda}^n(w).
\]
where $\lambda'$ is the composition of $n-1$ we obtain from $\lambda$ by deleting the box labeled by $n$ in $R(w)$.
\end{lemma}
\begin{proof} Recall that $v$ is the permutation whose one-line notation has $i$ in the $n$-th position and all remaining entries are placed in positions $1,2,\ldots, n-1$ in increasing order.  In particular, we obtain the base filling of the composition $\lambda'$ from the base filling of $\lambda$ by deleting the box containing $i$ and applying $v^{-1}$ to the remaining entries.  It follows immediately that $R(y)$ is the tableau of composition shape $\lambda'$ we obtain by deleting the box containing $n$ from $R(w)$ so $R(y)\in \RS(\lambda')$.  Thus $\inv_{\lambda'}(y) = \sqcup_{k=2}^{n-1} \inv_{\lambda}^k(w)$ as desired.  
\end{proof}

Motivated by the inductive formula from Lemma~\ref{lemma: Springer case induction}, we let $X_{\lambda'} \in \mathfrak{gl}_{n-1}(\C)$ be the nilpotent matrix defined as in~\eqref{eqn: X-definition} for the composition $\lambda'$ of $n-1$.  The proof of the lemma implies $X_{\lambda'}$ is the matrix corresponding to the linear transformation obtained by restricting $v^{-1}X_\lambda v$ to $\C^{n-1}\simeq\Span\{\mathbf{e}_1, \mathbf{e}_2, \ldots, \mathbf{e}_{n-1}\}$.

\begin{remark}\label{rem.restriction} Let  $\B^{X_{\lambda'}}$ denote the Spring fiber in $\Flags(\C^{n-1})$ corresponding to $X_{\lambda'}\in \mathfrak{gl}_{n-1}(\C)$.
The discussion above implies that for all $u_0\in U_0$, $vu_0yE_\bullet \in \B^{X_\lambda}$ if and only if $u_0yE_\bullet' \in \B^{X_{\lambda'}}$.
\end{remark}


\section{The $B_k(w)$-subgroups} \label{sec.Bmatrices}

We now introduce a collection of subgroups of $U$ associated to each $w\in S_n$ with $R(w)\in \RS(\lambda)$.  We use these subgroups in the next section to construct an affine paving of the Springer fiber $\B^{X}$ that restricts to a paving of the subvariety $\Hess(X,h)$.  Throughout this section, let $\lambda$ be a fixed weak composition of $n$ and $X_\lambda\in \mathcal{O}_{\Par(\lambda)}$ the matrix from Definition~\ref{eqn: X-definition} above.  

Suppose $wE_\bullet \in \B^{X_\lambda}$, or equivalently by Lemma~\ref{lemma: h-strict tab}, that $R(w)\in \RS(\lambda)$.  Using the factorization from~\eqref{eqn: factorization} we write $w=vy$ for $v=s_is_{i+1}\cdots s_{n-2}s_{n-1}$ where $i=w(n)=v(n)$.  Recall the morphism $\rho_1: C_w\to \C^{\ell(v)}$ defined in~\eqref{eqn.proj-morphism}.
The next lemma tells us that if $uwE_\bullet \in \B^{X_\lambda}$, then certain entries of $u_i=\rho_1(uwE_\bullet)$ must be zero.

\begin{lemma}\label{lemma: zeros} Suppose $uwE_\bullet \in \B^{X_\lambda}$ and $\rho_1(uwE_\bullet) = u_i\in U_i$ for $i=w(n)$.  Let $u_{ij}$ for $j>i$ denote the entry in the $i$-th row and $j$-th column of $u_i$.  Then $u_{ij}=0$ unless $j$ appears at the end of a row in the base filling $R(e)$.  In particular, the morphism of varieties $\rho_1: C_w\to \C^{\ell(v)}$ defined in~\eqref{eqn.proj-morphism} restricts to a morphism
\begin{eqnarray}\label{eqn.proj-morphism2}
\rho_1^{(\lambda)} : C_w\cap \B^{X_\lambda} \to \C^{d_n}
\end{eqnarray}
where $d_n:=|\inv_\lambda^n(w)|$.
\end{lemma}
\begin{proof} By Lemma~\ref{lemma: Schubert cell decomp} we may write $uwE_\bullet = u_ivu_0yE_\bullet$ for some $u_0\in U_0$ and $u_i = \rho_1(uwE_\bullet)$. Suppose $j>i$ does not fill a box at the end of a row in the base filling $R(e)$ of $\lambda$.  This implies $\mathbf{e}_j\in \im(X_\lambda)$.  Let $V_\bullet=(\mathbf{v}_1\mid \mathbf{v}_2\mid \ldots\mid \mathbf{v}_n)$ where $\mathbf{v}_k = uw(\mathbf{e}_k)$.  Since $V_\bullet \in \B^{X_\lambda}$ we must have $\im(X_\lambda) \subseteq V_{n-1}$ and so $\mathbf{e}_j\in \Span\{\mathbf{v}_1, \ldots, \mathbf{v}_{n-1}\}$. Thus $\mathbf{e}_j=\sum_{k=1}^{n-1} c_k \mathbf{v}_k$ for some $c_1, \ldots, c_{n-1}\in \C$.   Applying $(u_iv)^{-1}$ to both sides we obtain
\begin{eqnarray*}
v^{-1}(\mathbf{e}_{j}-u_{ij}\mathbf{e}_i)=  \sum_{k=1}^{n-1} c_ku_0y(\mathbf{e}_k) &\Rightarrow& \mathbf{e}_{v^{-1}(j)} -u_{ij} \mathbf{e}_n = \sum_{k=1}^{n-1} c_ku_0\mathbf{e}_{y(k)}\\
&\Rightarrow& -u_{ij}\mathbf{e}_n = -\mathbf{e}_{v^{-1}(j)} + \sum_{k=1}^{n-1} c_ku_0\mathbf{e}_{y(k)}
\end{eqnarray*}
The RHS of the above equation is in $\Span\{\mathbf{e}_1, \ldots, \mathbf{e}_{n-1}\}$, implying $u_{ij}=0$ as desired.  Finally, we note that $j>i$ and $j$ appears at the end of a row in the base filling $R(e)$ if and only if $(w^{-1}(i), w^{-1}(j)) = (n, w^{-1}(j))$ is a Springer inversion of $R(w)$ in $\inv_\lambda^n(w)$.  This shows 
\[
\inv_\lambda^n(w) = \{ (n, w^{-1}(j))  \mid i<j \, \textup{ and $j$ labels a box at the end of a row in $R(e)$}   \}
\]  
Thus if $uwE_\bullet\in \B^{X_\lambda}$ we get 
\[
\rho_1(uwB) =u_i = I_n + \sum_{\substack{i<j\\ (n,w^{-1}(j)) \in \inv_\lambda^n(w)}} u_{ij}E_{ij}.
\] 
This yields the description of the restriction of $\rho_1$ to $C_w\cap \B^{X_\lambda}$ in~\eqref{eqn.proj-morphism2}.
\end{proof}

The goal of the next section is to construct a generic element of $C_w\cap \B^{X_\lambda}$ whenever this intersection is nonempty. We do so by introducing a collection of subgroups of $U$ associated to each $R(w)\in \RS(\lambda)$.  Recall that $\inv_\lambda(w)$ denotes the set of Hessenberg inversions corresponding to $h=(0,1,2,\ldots, n-1)$, namely the Springer inversions.

\begin{definition}\label{defn: B-matrices} Let $2\leq k \leq n$.  We define $B_k(w)$ to be the set of all matrices $g_k$ such that:
\begin{enumerate}
\item if $(k,\ell)\in \inv_{\lambda}^k(w)$ and $\mathbf{e}_{w(j)} = X_\lambda^{m}\mathbf{e}_{w(\ell)}$ for some $m\in \Z_{\geq0}$ then 
\[
g_k\mathbf{e}_{w(j)} = \mathbf{e}_{w(j)}+x_{w(k)w(\ell)}X_\lambda^m\textbf{e}_{w(k)} \textup{ for some $x_{w(k)w(\ell)}\in \C$},
\]
 \item and $g_k \mathbf{e}_{w(j)} = \mathbf{e}_{w(j)}$ otherwise. 
 \end{enumerate}
\end{definition}

From the definition above, we see that each element  $g_k$ of $B_k(w)$ is uniquely determined by the values of $(x_{w(k) w(\ell_1)} , \ldots, x_{w(k) w(\ell_d)})$ for $\inv_\lambda^k(w) = \{(k,\ell_1),  \ldots, (k,\ell_d)\}$ and $d=d_k$. To emphasize this, we sometimes write $g_k = g_k(x_{w(k)w(\ell_1)} , \ldots, x_{w(k)w(\ell_d)})$ and say that $(x_{w(k) w(\ell_1)} , \ldots, x_{w(k) w(\ell_d)})$ are the coordinates of $g_k$.

\begin{example} Let $n=7$ and $\lambda=(3,2,2)$.  We consider $w=[3,2,6,1,7,4,5]$ with corresponding tableau $R(w)\in \RS(\lambda)$ as shown below; the base filling $R(e)$ is also below.
\[\ytableausetup{centertableaux}
R(w) =
\begin{ytableau}
1 & 3 & 5 \\
2 & 7 \\
4 & 6 
\end{ytableau} 
\quad\quad\quad
R(e) =  \begin{ytableau}
3 & 6 & 7 \\
2 & 5 \\
1 & 4 
\end{ytableau} 
\]
In this case, $\inv_{\lambda}(w) = \{ (7,5), (6,5), (4,2), (4,3), (2,1) \}$ so, in particular, $\inv_{\lambda}^4(w) = \{ (4,2), (4,3) \}$.  We have $(w(4), w(2)) = (1,2)$ and $(w(4), w(3)) = (1,6)$.  Since $X_\lambda \mathbf{e}_{w(4)} = \mathbf{0}$, the elements of $B_4(w)$ are matrices of the form
\[
I_7 + x_{12}E_{12} + x_{16}E_{16}
\textup{ where } x_{12}, x_{16}\in \C.
\]
For another example, consider $\inv_\lambda^6(w)=\{(6,5)\}$; we have $(w(6), w(5)) = (4,7)$.  In this case, $X_\lambda \mathbf{e}_{w(6)}=X_\lambda \mathbf{e}_4 = \mathbf{e}_1$ and $X_\lambda^2\mathbf{e}_{w(6)} = X^2 \mathbf{e}_4=\mathbf{0}$.  Therefore the elements of $B_6(w)$ are matrices of the form
\[ 
I_7+x_{47}E_{47}+x_{47}E_{16}
\textup{ where } x_{47}\in \C.
\]
\end{example}

Let $2\leq k \leq n$ and $g_k\in B_k(w)$. Then $(g_k)_{aa}=1$ for all $a\in [n]$.  Suppose $(g_k)_{ab}\neq0$ for $a\neq b$. By definition, $(a,b) = (w(\ell'), w(k'))$ where $\ell',k'\in[n]$ such that $X^m e_{w(\ell)} = e_{w(\ell')}$ and $X^me_{w(k)} = e_{w(k')}$ for some $m\in \Z_{\geq0}$.  Since the action of $X_\lambda$ on the standard basis vectors is determined by the base filling $R(e)\in \RS(\lambda)$, it follows that $w(k')$ fills the $m$-th box to the left of $w(k)$ in $R(e)$ and $w(\ell')$ fills the $m$-th box to the left of $w(\ell)$ in $R(e)$.  Equivalently, $k'$ fills the $m$-th box to the left of $k$ in $R(w)$ and $\ell'$ fills the $m$-th box to the left of $\ell$ in $R(w)$.   

Since $(k,\ell)\in \inv_\lambda(w)$, we know that the box labeled by $k$ appears below the box labeled by $\ell$ and in the same column or in any column strictly to the left of $\ell$ in $R(w)$.  Therefore the same must be true of the pair $(k',\ell')$, i.e., the box labeled by $k'$ appears below the box labeled by $\ell'$ and in the same column or in any column strictly to the left of $\ell'$ in $R(w)$; and similarly for the boxes labeled by $w(k')$ and $w(\ell')$ in $R(e)$. By definition of the base filling $R(e)$, we conclude that $a=w(k')<w(\ell')=b$ so $g_k\in U$.

We summarize the discussion above in the following remark.

\begin{remark} Let $2\leq k \leq n$ and $w\in S_n$ such that such that $R(w)\in \RS(\lambda)$.  Then $B_k(w)\subseteq U$ and furthermore, given $g_k\in B_k$ we have $(g_k)_{ab}\neq 0$ if and only if $a=b$ \textup{(}in which case $(g_k)_{aa}=1$\textup{)} or $(a,b) = (w(\ell'), w(k'))$ where $k'$ fills the $m$-th box to the left of $k$ in $R(w)$ and $\ell'$ fills the $m$-th box to the left of $\ell$ in $R(w)$ for some $m\in \Z_{\geq 0}$ and $(k,\ell)\in \inv_\lambda^k(w)$.  
\end{remark}

The remainder of this section contains results describing the structure of the matrices in $B_k(w)$; in most cases, our proofs consist of straightforward computations using linear algebra.

\begin{lemma}\label{lemma: abelian subgroup} Let $2\leq k \leq n$ and $w\in S_n$ such that $R(w)\in \RS(\lambda)$.  The set of matrices $B_k(w)$ from Definition~\ref{defn: B-matrices} is an abelian subgroup of $U$ of dimension $|\inv_\lambda^k(w)|$.  
\end{lemma}
\begin{proof} Let $g_k, h_k\in B_k(w)$, where $g_k$ has coordinates $(x_{w(k)w(\ell_1)} , \ldots, x_{w(k)w(\ell_d)})$ and $h_k$ has coordinates $(y_{w(k)w(\ell_1)} , \ldots, y_{w(k)w(\ell_d)})$ for $d=d_k$.  We will prove $g_kh_k$ is the matrix in $B_k(w)$ with coordinates $(x_{w(k)w(\ell_1)}+y_{w(k)w(\ell_1)} , \ldots, x_{w(k)w(\ell_d)}+y_{w(k)w(\ell_d)})$.  In other words, we prove that if $(k,\ell)\in \inv_\lambda^k(w)$ and $\mathbf{e}_{w(j)} = X_\lambda^m \mathbf{e}_{w(\ell)}$ for some $m\in \Z_{\geq0}$, then 
\begin{eqnarray} \label{eqn: composition formula}
(g_kh_k)\mathbf{e}_{w(j)} = \mathbf{e}_{w(j)} + (x_{w(k)w(\ell)}+y_{w(k)w(\ell)})X_\lambda^m \mathbf{e}_{w(k)} 
\end{eqnarray}
and $(g_kh_k)\mathbf{e}_{w(j)}=\mathbf{e}_{w(j)}$ otherwise.  It follows directly from this formula that $B_k(w)$ is an abelian subgroup of $U$.

Now consider the action of $g_kh_k$ on any basis element $\mathbf{e}_{w(j)}$.  If $w(j)$ does not appear in $R(e)$ to the left and in the same row as any $w(\ell)$ for $(k,\ell)\in \inv^k_\lambda(w)$, that is, if $\mathbf{e}_{w(j)}\neq X_\lambda^m\mathbf{e}_{w(\ell)}$ for some $m\in \Z_{\geq0}$, then both $g_k$ and $h_k$ map $\mathbf{e}_{w(j)}$ to itself and $(g_kh_k)\mathbf{e}_{w(j)}=\mathbf{e}_{w(j)}$.

If $w(j)$ is to the left and in the same row of $R(e)$ as $w(\ell)$ for some $(k,\ell)\in \inv_\lambda^k(w)$, then $g_k \mathbf{e}_{w(j)} = \mathbf{e}_{w(j)} + x_{w(k) w(\ell)}X_\lambda^m\mathbf{e}_{w(k)}$ and $h_k\mathbf{e}_{w(j)} = \mathbf{e}_{w(j)} + y_{w(k)w(\ell)}X_\lambda^m \mathbf{e}_{w(k)}$ for some $m\in\Z_{\geq 0}$. Write $X_\lambda^m \mathbf{e}_{w(k)}= \mathbf{e}_{w(k')}$ where $w(k')$ labels the $m$-th box to the left of $w(k)$ in $R(e)$.  Then $k'$ labels the $m$-th box to the left of $k$ in $R(w)$ and hence $k'$ is not in the same row as $\ell$ for any $(k,\ell)\in \inv^k_\lambda(w)$ by Remark~\ref{rem: k-inversions}.  This implies $g_k\mathbf{e}_{w(k')} = \mathbf{e}_{w(k')} = h_k\mathbf{e}_{w(k')}$ and formula~\eqref{eqn: composition formula} now follows.  The assertion that the dimension of this subgroup is $|\inv_\lambda^k(w)|$ is clear from the definition.
\end{proof}

\begin{lemma}\label{lemma: action of Bk} Let $2\leq k \leq n$ and $w\in S_n$ such that $R(w)\in \RS(\lambda)$.  For all $g_k\in B_k(w)$ we have
\[
g_k(\Span\{\mathbf{e}_{w(1)}, \mathbf{e}_{w(2)}, \cdots, \mathbf{e}_{w(k)}\}) = \Span\{\mathbf{e}_{w(1)}, \mathbf{e}_{w(2)}, \cdots, \mathbf{e}_{w(k)}\}
\]
and $g_k\mathbf{e}_{w(j)} = \mathbf{e}_{w(j)}$ for all $j\geq k$.  Furthermore, $g_k(\ker(X_\lambda))\subseteq \ker(X_\lambda)$.
\end{lemma}
\begin{proof} By definition, if $\mathbf{e}_{w(j)} = X_\lambda^m \mathbf{e}_{w(\ell)}$ for some $m\in \Z_{\geq0}$ and $(k,\ell)\in \inv_\lambda^k(w)$ then $w(j)$ fills the $m$-th box to the left of $w(\ell)$ in $R(e)$.  Therefore $j$ fills the $m$-th box to the left of $\ell$ in $R(w)$ and since $R(w)\in \RS(\lambda)$ we have $j\leq \ell < k$.  This shows $g_k\mathbf{e}_{w(j)} = \mathbf{e}_{w(j)}$ for all $j\geq k$.  Similarly, if $X_\lambda^m \mathbf{e}_{w(k)} = \mathbf{e}_{w(k')}$ then $k'$ fills the $m$-th box to the left of $k$ in $R(w)$ and $k'\leq k$.  The first assertion of the lemma now follows.

To prove the second, suppose $\mathbf{e}_{w(j)} \in \ker(X_\lambda)$ with $\mathbf{e}_{w(j)} = X_\lambda^m \mathbf{e}_{w(\ell)}$ for some $m\in \Z_{\geq 0}$ and $(k,\ell)\in \inv_\lambda^k(w)$.  We must show $g_k\mathbf{e}_{w(j)}\in \ker(X_\lambda)$, or equivalently, if $X_\lambda^m \mathbf{e}_{w(k)}\neq \mathbf{0}$ then $X_\lambda^m\mathbf{e}_{w(k)}\in \ker(X_\lambda)$.  Since $g_k\in U$, if $X_\lambda^m \mathbf{e}_{w(k)}\neq \mathbf{0}$ then $X_\lambda^m \mathbf{e}_{w(k)}= \mathbf{e}_{w(k')}$ where $w(k')< w(j)$.  The assumption that $\mathbf{e}_{w(j)}\in \ker(X_\lambda)$ implies $w(j)$ fills a box in the first column of $R(e)$ and the fact that $w(k')<w(j)$ implies that $w(k')$ fills a box below $w(j)$ in the first column of $R(e)$ so $X_\lambda \mathbf{e}_{w(k')}=\mathbf{0}$ as desired.
\end{proof}

The next lemma shows that the matrices in $B_k(w)$ for $k<n$ \textit{almost} commute with $X_\lambda$ and the elements of $B_n(w)$ always commute with $X_\lambda$.

\begin{lemma}\label{lemma: commutator step1} Suppose $R(w)\in \RS(\lambda)$ and let $2\leq k \leq n$.  Given $\ell\in [n]$, let $r_\ell$ denote the label of the box directly to the right of the box labeled by $\ell$ in $R(w)$, if such a box exists.  Then for all $g_k\in B_k(w)$,
\begin{eqnarray}\label{eqn: commutator step1}
(g_kX_\lambda - X_\lambda g_k)\mathbf{e}_{w(j)} = \left\{ \begin{array}{ll} x_{w(k)w(\ell)}\mathbf{e}_{w(k)} & \textup{ if } j=r_{\ell},\, (k, \ell)\in \inv_\lambda^k(w) \\ 0 & \textup{ otherwise}  \end{array}\right.
\end{eqnarray}
where $x_{w(k)w(\ell)}\in \C$ for $(k,\ell)\in \inv_\lambda^k(w)$ is a coordinate of $g_k\in B_k(w)$.  In particular, if $k=n$ this formula shows that $X_\lambda$ and $g_n$ commute.
\end{lemma}
\begin{proof} First, suppose $j=r_\ell$ for some $\ell\in [n]$ such that $(k,\ell)\in \inv_\lambda^k(w)$.  By definition, this is the case if and only if $X_\lambda \mathbf{e}_{w(j)}=\mathbf{e}_{w(\ell)}$ and $k\leq j$.  In particular, Lemma~\ref{lemma: action of Bk} tells us $g_k\mathbf{e}_{w(j)}=\mathbf{e}_{w(j)}$.  We now have:
\begin{eqnarray*}
(g_k X_\lambda - X_\lambda g_k) \mathbf{e}_{w(j)} = g_k \mathbf{e}_{w(\ell)} - \mathbf{e}_{w(\ell)} = x_{w(k)w(\ell)}\mathbf{e}_{w(k)}
\end{eqnarray*}
where the last equation follows directly from Definition~\ref{defn: B-matrices}.

It remains to show that if $j\neq r_\ell$ for any $\ell\in [n]$ such that $(k,\ell)\in \inv_\lambda^k(w)$, then $g_kX_\lambda \mathbf{e}_{w(j)} = X_\lambda g_k\mathbf{e}_{w(j)}$.  This is straightforward to prove using Definition~\ref{defn: B-matrices}, so we omit the details.

Finally, when $k=n$ we have $(n,\ell)\in \inv_\lambda(w)$ only if $\ell$ labels a box at the end of a row in $R(e)$.  Thus~\eqref{eqn: commutator step1} tells us $X_\lambda$ and $g_n$ commute.
\end{proof}

Our goal is to use induction to analyze the flags in $C_w\cap \B^{X_\lambda}$.  In particular, the next statement shows that we can naturally identify the group $B_k(w)$ for $k<n$ with a subgroup of the same type in $GL_{n-1}(\C)$ using the map from~\eqref{eqn.induction}.

\begin{proposition}\label{lemma: B_k for  kneq n} Suppose $w\in S_n$ such that $R(w)\in \RS(\lambda)$ and $w=vy$ is the factorization defined in~\eqref{eqn: factorization}, with  $i=w(n)$.  
\begin{enumerate}
\item For all $2\leq k \leq n-1$, we have $v^{-1}B_k(w) v= B_{k}(y)$ where $B_{k}(y)$ is the subgroup of $U_0\subseteq GL_{n-1}(\C)$ corresponding to $R(y)\in \RS(\lambda')$.

\item  Each $g_n\in B_n(w)$ can be factored uniquely as $g_n=u_i b_n$ where $u_i\in U_i$, $b_n\in v U_0 v^{-1}$, and the $i$-th row of $u_i$ is equal to the $i$-th row of $g_n$.  In particular, $u_i = \rho_1^{(\lambda)}(g_n wE_\bullet)$ and there exists an isomorphism of varieties $\C^{d_n} \to B_n(w)$ with inverse given by $g_n \mapsto \rho_1^{(\lambda)}(g_nwE_\bullet)$.

\end{enumerate}
\end{proposition}

\begin{proof} The proof of Lemma~\ref{lemma: Springer case induction} shows that $R(y)$ is the row-strict tableau we obtain by deleting the box labeled by $n$ from $R(w)$.  This lemma also tells us that $R(y)\in \RS(\lambda')$ where $\lambda'$ is the composition of $n$ corresponding to $R(y)$, and for all $k\leq n-1$, $(k,\ell)\in \inv_{\lambda'}(y)$ if and only if $(k,\ell)\in \inv_\lambda(w)$.  Recall that $X_{\lambda'}$ is given by the restriction of $v^{-1}X_\lambda v$ to $\C^{n-1}= \Span\{\mathbf{e}_1, \ldots, \mathbf{e}_{n-1}\}$.

Let $g_k\in B_k(w)$ have coordinates $(x_{w(k)w(\ell_1)}, \ldots, x_{w(k)w(\ell_{d})})$ where $d=d_k$.  We prove $v^{-1}g_k v\in B_k(y)$ by showing $v^{-1}g_k v$ acts on the basis vectors $\{\mathbf{e}_1, \ldots, \mathbf{e}_{n-1}\}$ by the formulas given in Definition~\ref{defn: B-matrices} where $\lambda$ is replaced by $\lambda'$ and $w$ is replaced by $y$.  Suppose $(k,\ell)\in \inv_{\lambda'}(y)$ and $\mathbf{e}_{y(j)} = X_{\lambda'}^m \mathbf{e}_{y(\ell)}$ for some $m\in \Z_{\geq0}$.  By definition of $X_{\lambda'}$, this is the case if and only if $\mathbf{e}_{w(j)} = X_\lambda^m \mathbf{e}_{w(\ell)}$ for $(k,\ell)\in \inv_\lambda^k(w)$. Thus $v^{-1}g_kv$ is the matrix such that:
\begin{eqnarray*}
(v^{-1}g_kv)(\mathbf{e}_{y(j)})&=& v^{-1}(g_k \mathbf{e}_{w(j)}) =  v^{-1}(\mathbf{e}_{w(j)} + x_{w(k)w(\ell)}X_\lambda^m \mathbf{e}_{w(k)})\\ &=& \mathbf{e}_{y(j)} + x_{w(k)w(\ell)}X^m_{\lambda'} \mathbf{e}_{y(k)} .
\end{eqnarray*}
By similar reasoning, we have $(v^{-1}g_kv)\mathbf{e}_{y(j)}=\mathbf{e}_{y(j)}$ if $\mathbf{e}_{y(j)}\neq X_{\lambda'}^m \mathbf{e}_{y(\ell)}$ for any $(k,\ell)\in \inv_{\lambda'}(y)$ and $m\in \Z_{\geq 0}$. This proves $v^{-1}B_k(w)v\subseteq B_k(y)$, after relabeling the coordinate $x_{w(k)w(\ell)}$ by $x_{y(k)y(\ell)}$. The proof that $B_k(y)\subseteq v^{-1}B_k(w)v$ follows in exactly the same way.

To prove statement (2), suppose $g_n\in B_n(w)$ with coordinates $(x_{w(n) w(\ell_1)}, \ldots, x_{w(n) w(\ell_d)}) = (x_{i w(\ell_1)}, \ldots, x_{i w(\ell_d)})$ where $d=d_n$.  Let $u_i$ be the matrix defined uniquely by the equation
\[
u_i (\mathbf{e}_{w(\ell)}) = \left\{ \begin{array}{ll} \mathbf{e}_{w(\ell)} + x_{i w(\ell)} \mathbf{e}_{i} & \textup{ if $(n,\ell) \in \inv_\lambda^n(w)$}\\ \mathbf{e}_{w(\ell)} & \textup{ otherwise. } \end{array}\right. 
\]
Then $u_i\in U_i$ and by definition $u_i$ has $i$-th row equal to $g_n$.  Furthermore, $b_n$ is uniquely determined by the formula
\[
b_n (\mathbf{e}_{w(j)}) = \left\{ \begin{array}{ll} \mathbf{e}_{w(j)} + x_{i w(\ell)} X_\lambda^m \mathbf{e}_{i} & \textup{ if $(n,\ell) \in \inv_\lambda^n(w)$, $\mathbf{e}_{w(j)} = X_\lambda^m e_{w(\ell)}$ for some $m\geq 1$}\\ \mathbf{e}_{w(j)} & \textup{ otherwise. } \end{array}\right.
\]
for all $m\geq 1$.

By the previous paragraph, the matrix $u_i$ is uniquely determined by, and uniquely determines, $g_n$. We now show that $b_n:= u_i^{-1}g_n\in vU_0v^{-1}$. It is straightforward to see from the definition of $u_i$ and $g_n$ that  $u_i^{-1}g_n$ has the same entries as $g_n$ except for the $i$-th row, which is $0$ in all off-diagonal entries.  Thus $u_i^{-1}g_n$ may only have a nonzero off-diagonal entry in positions $(w(n'), w(\ell'))$ with $w(n')<w(\ell')$ where $n'$ labels the $m$-th box to the left of $n$ in $R(w)$ and $\ell'$ labels the $m$-th box to the left of $\ell$ in $R(w)$ for some $m\in \Z_{\geq1}$.  In particular, $n'\neq n$.  It follows that $v^{-1}(u_i^{-1}g_n)v$ may only have nonzero off-diagonal entries in positions $(y(n'), y(\ell'))$ where $y(n')\neq n$, $y(\ell')\neq n$, and $y(n')<y(\ell')$ so $b_n\in vU_0v^{-1}$ as desired.

Finally, the map $\C^{d_n} \to B_n(w)$ defined by $(x_{iw(\ell_1)}, \ldots, x_{iw(\ell_d)}) \mapsto g_n(x_{iw(\ell_1)}, \ldots, x_{iw(\ell_d)})$ is clearly a morphism of varieties. By the above, we may write $g_n = u_ib_n$.  Since $g_n$ commutes with $X_\lambda$ we get $g_nwE_\bullet \in \B^{X_\lambda}$ and:
\[
g_n w E_\bullet = u_i v u_0' y E_\bullet\; \textup{ where } \; u_0'=v^{-1}b_n v\in U_0, u_i\in U_i. 
\]
This shows that $\rho_1^{(\lambda)}(g_nwE_\bullet) = u_i$, where $\rho_1^{(\lambda)}$ is the morphism defined in Lemma~\ref{lemma: zeros}.
The inverse map $B_n(w) \to \C^{d_n}$ is precisely $g_n \mapsto u_i = \rho_1^{(\lambda)}(g_nwE_\bullet)$.  This proves $\C^{d_n} \to B_n(w)$ is an isomorphism of varieties.
\end{proof}

\begin{example}\label{ex: induction} We illustrate the previous proposition with an example.  Let $n=6$ and $\lambda = (2,2,2)$ and $w=[3,6,2,1,5,4]$.  In this case, $v=[1,2,3,5,6,4]$ and $y=[3,5,2,1,4,6]$.  
\[\ytableausetup{centertableaux}
R(w) =
\begin{ytableau}
1 & 2 \\
3 & 5 \\
4 & 6 
\end{ytableau} 
\quad\quad\quad
R(e) =  \begin{ytableau}
3 & 6 \\
2 & 5 \\
1 & 4
\end{ytableau} 
\]
We have $\inv_\lambda^6(w) = \{(6,5), (6,2)\}$ where $(w(6), w(5)) = (4,5)$ and $(w(6), w(2)) = (4,6)$. An arbitrary element of $B_6(w)$ is of the form:
\begin{eqnarray*}
g_6 &=& I_6 + x_{45}(E_{45}+E_{12}) + x_{46}(E_{46}+E_{13}) \\ &=& (I_6+ x_{45}E_{45}+x_{46}E_{46})(I_6+ x_{45}E_{12}+x_{46}E_{13}) = u_4 b_6
\end{eqnarray*}
where $x_{45},x_{46}\in \C$ are the coordinates of $g_6$, $u_4= I_6+ x_{45}E_{45}+x_{46}E_{46}$ and $b_6= I_6+ x_{45}E_{12}+x_{46}E_{13}$. Note that $v^{-1}b_6v=b_6\in U_0$, confirming statement (2) of Proposition~\ref{lemma: B_k for  kneq n}.  The isomorphism $\C^2 \to B_6(w)$ is defined by 
\[
(x_{45}, x_{46}) \mapsto (I_6+ x_{45}E_{45}+x_{46}E_{46})(I_6+ x_{45}E_{12}+x_{46}E_{13}) = g_6(x_{45}, x_{46}).
\]

Now consider $\inv_\lambda^4(w) = \{(4,2), (4,3)\}$; we have $(w(4),w(2)) = (1, 6)$ and $(w(4), w(3)) = (1,2)$.  An arbitrary element of $B_4(w)$ is of the form:
\[
g_4 = I_6+ x_{12}E_{12} + x_{16}E_{16}\]
where $x_{12}, x_{16}$ are the coordinates of $g_4$. We have $v^{-1}g_4v = I_6 + x_{12}E_{12} + x_{16}E_{15}\in B_4(y)\subset U_0$.  Since $B_4(y) = \{I_5 + x_{12}E_{12} + x_{15}E_{15} \mid x_{12}, x_{15} \}$ we obtain $v^{-1}B_4(w)v = B_4(y)$ via the identification from~\eqref{eqn.induction}, confirming statement (1) of Proposition~\ref{lemma: B_k for  kneq n}.
\end{example}


\section{An affine paving}\label{sec.paving}

In this section, we prove our first main theorem.  In particular, Theorem~\ref{thm: main thm2} below tells us that $\Hess(X_\lambda, h)$ is paved by affines for all $h\in \mathcal{H}$.   As noted above, this result is not new, but our constructive methods are elementary and provide more insight into the structure of the paving.   We obtain our paving of $\Hess(X_\lambda,h)$ by restricting an affine paving for the Springer fiber $\B^{X_\lambda}$.  As a consequence, we recover Tymoczko's formulas for the dimension of each affine cell, originally proved in~\cite[Theorem 7.1]{Tymoczko2006} for Hessenberg varieties in the flag variety of $GL_n(\C)$ corresponding to Hessenberg functions such that $h(i)\geq i$ for all $i$.

We begin by constructing an affine paving of the Springer fiber $\B^{X_\lambda}$ using the subgroups $B_k(w)$ introduced above.  Throughout this section, $\lambda$ denotes a fixed composition of $n$.  The next proposition is the lynchpin of our inductive arguments; our proof resembles that of \cite[Lemma 1]{Spaltenstein1976} (see also \cite[\S 11.3]{Jantzen}).

\begin{proposition}\label{prop.isom}  Suppose $R(w)\in \RS(\lambda)$ and $w=vy$ is the factorization from~\eqref{eqn: factorization} with $i=w(n)$.  Let $\lambda'$ be the composition of $n-1$ obtained from $\lambda$ by deleting the box labeled by $n$ in $R(w)$ and $X_{\lambda'}$ denote the restriction of $v^{-1}X_\lambda v$ to $\C^{n-1}=\Span\{\mathbf{e}_1, \mathbf{e}_2, \ldots, \mathbf{e}_{n-1}\}$.  There exists an isomorphism of varieties 
\[
\varphi_w: \C^{d_n} \times (C_y\cap \B^{X_{\lambda'}}) \to C_w\cap \B^{X_\lambda}
\]
where $d_n=|\inv_\lambda^n(w)|$.
\end{proposition}
\begin{proof} We use the isomorphism $\C^{d_n}\to B_n(w)$ from Proposition~\ref{lemma: B_k for  kneq n}(2) throughout the proof.  Given a flag $V_\bullet'\in \B^{X_{\lambda'}}$ we write $V_\bullet' = u_0yE_\bullet'$ for some $u_0\in U_0$.  It suffices to show that 
\[
\varphi_w: B_n(w) \times (C_y\cap \B^{X_\lambda'}) \to C_w\cap \B^{X_\lambda},\; (g_n, u_0yE_\bullet')\mapsto g_nvu_0yE_\bullet
\]
is an isomorphism of varieties. Note first that $\varphi_w$ is well defined since $g_n$ commutes with $X_\lambda$ by Lemma~\ref{lemma: commutator step1} and $vu_0yE_\bullet\in \B^{X_\lambda}$ by Remark~\ref{rem.restriction}.

To begin, we argue that $\varphi_w$ is a bijection.  Suppose first that $g_nvu_0yE_\bullet = h_n v u_0' yE_\bullet$ for some $g_n,h_n \in B_n(w)$ and $u_0,u_0'\in U_0$. By Proposition~\ref{lemma: B_k for kneq n}(2), we may write $g_n =u_ib_n$ for $u_i \in U_i$ such that the $i$-th row of $u_i$ is equal to the $i$-th row of $g_n$ and $b_n\in vU_0v^{-1}$.  Similarly, we have $h_n = u_i' b_n'$ for $u_i' \in U_i$ such that the $i$-th row of $u_i'$ is equal to the $i$-th row of $h_n$ and $b_n'\in vU_0v^{-1}$.   Thus
\[
u_i = \rho_1^{(\lambda)}(g_n vu_0 E_\bullet) = \rho_1^{(\lambda)}(h_n vu_0' E_\bullet) = u_i'
\]
where $\rho_1^{(\lambda)}$ is the map defined in~\eqref{eqn.proj-morphism2} above. Proposition~\ref{lemma: B_k for kneq n}(2) now implies $g_n=h_n$ and consequently $u_0yE_\bullet =  u_0'yE_\bullet$ as well.  This shows that $\varphi_w$ is injective.

To prove surjectivity, let $uwE_\bullet\in C_w\cap \B^{X_\lambda}$ for some $u\in U^w$.  Using Lemma~\ref{lemma: Schubert cell decomp} we may write $uwE_{\bullet} = u_ivu_0yE_\bullet$ for $u_0\in U_0$ and $u_i\in U_i$.  By Lemma~\ref{lemma: zeros}, we have that $\rho_1^{(\lambda)}(uwE_\bullet) = u_i \in \C^{d_n}$. Let $g_n\in B_n(w)$ have coordinates determined by $\rho_1^{(\lambda)}(uwE_\bullet)$ via the isomorphism from Proposition~\ref{lemma: B_k for kneq n}(2).  Thus we have $g_n= u_i b_n$ where $u_i = \rho_1^{(\lambda)}(uwE_\bullet)$ and $b_n \in  vU_0v^{-1}$.   Now 
\[
g_n^{-1} uwE_\bullet = v  u_0'yE_\bullet \in \B^{X_\lambda}\,  \textup{ where } \, u_0' = v^{-1}b_nv u_0 \in U_0.  
\]
This shows that $\varphi_w(g_n, u_0'yE_\bullet') = uwE_\bullet$, so $\varphi_w$ is also surjective.

The bijection $\varphi_w$ is in fact a morphism of varieties since $\C^{d_n} \to B_n(w)$ and the action map $GL_{n-1}(\C) \times \Flags(\C^n) \to \Flags(\C^n)$ are morphisms of varieties.  The inverse map, defined by $\varphi_w^{-1} (uwE_\bullet) = ( g_n ,  v^{-1}g_n^{-1}uwE_{\bullet}')$ where $g_n\in B_n(w)$ has coordinates determined by $\rho_1^{(\lambda)}(uwE_\bullet)$, is also a morphism of varieties. This concludes the proof.
\end{proof}

We now introduce flags constructed using the subgroups $B_k(w)$ of the previous section.  We use these flags to give an explicit description of our affine paving of the Springer fibers below.

\begin{definition}\label{def.Dset} Let $w\in S_n$ such that $R(w)\in \RS(\lambda)$.  We define a subset $\mathcal{D}_w\subseteq C_w$ inductively as follows. First, set $\mathcal{D}_w^1 = \{w(E_\bullet)\}$. Then for each $2\leq k \leq n$ let 
\[
\mathcal{D}^k_w = \{g_k(V_\bullet) \mid g_k\in B_k(w) \textup{ and } V_\bullet\in \mathcal{D}^{k-1}_w  \}.
\]
Let $\mathcal{D}_w := \mathcal{D}_w^n$. In other words, $bwE_\bullet \in \mathcal{D}_w$ if and only if $b=g_ng_{n-1}\ldots g_2$ with $g_k\in B_k(w)$ for all~$k$.
\end{definition}

The definition of $\mathcal{D}_w$ depends on our choice of composition $\lambda$ (since $R(w)\in \RS(\lambda)$ and the definition of $B_k(w)$ depends on $R(w)$), but we suppress this dependence in the notation for $\mathcal{D}_w$.  Note that for any composition, we have $\mathcal{D}_e = \{E_\bullet\}$. 

\begin{example}\label{ex: Dw set} 
Continuing Example~\ref{ex: induction}, let $n=6$ and $\lambda = (2,2,2)$ and $w=[3,6,2,1,5,4]$.  
\[\ytableausetup{centertableaux}
R(w) =
\begin{ytableau}
1 & 2 \\
3 & 5 \\
4 & 6 
\end{ytableau} 
\quad\quad\quad
R(e) =  \begin{ytableau}
3 & 6 \\
2 & 5 \\
1 & 4
\end{ytableau} 
\]
In this case, $\inv_\lambda(w) = \{(6,5), (6,2), (5,2),(4,3),(4,2),(3,2)\}$.  Computing the $B_k(w)$ subgroups, we have that $B_2(w)=\{I_6\}$ and: 
\begin{eqnarray*}
 B_{3}(w) = \left\{\begin{bmatrix}
1 & 0 & 0 & 0 & 0 & 0 \\
0 & 1 & 0 & 0 & 0 & x_{26} \\
0 & 0 & 1 & 0 & 0 & 0 \\
0 & 0 & 0 & 1 & 0 & 0 \\
0 & 0 & 0 & 0 & 1 & 0 \\
0 & 0 & 0 & 0 & 0 & 1 \\
 \end{bmatrix} \right\},\;\;\;
  B_{4}(w) = \left\{\begin{bmatrix}
1 & x_{12} & 0 & 0 & 0 & x_{16} \\
0 & 1 & 0 & 0 & 0 & 0 \\
0 & 0 & 1 & 0 & 0 & 0 \\
0 & 0 & 0 & 1 & 0 & 0 \\
0 & 0 & 0 & 0 & 1 & 0 \\
0 & 0 & 0 & 0 & 0 & 1 \\
 \end{bmatrix}\right\}, \\
 B_{5}(w) = \left\{\begin{bmatrix}
1 & 0 & 0 & 0 & 0 & 0 \\
0 & 1 & x_{56} & 0 & 0 & 0 \\
0 & 0 & 1 & 0 & 0 & 0 \\
0 & 0 & 0 & 1 & 0 & 0 \\
0 & 0 & 0 & 0 & 1 & x_{56} \\
0 & 0 & 0 & 0 & 0 & 1 \\
 \end{bmatrix}\right\},\;\;
  B_{6}(w) = \left\{\begin{bmatrix}
1 & x_{45} & x_{46} & 0 & 0 & 0 \\
0 & 1 & 0 & 0 & 0 & 0 \\
0 & 0 & 1 & 0 & 0 & 0 \\
0 & 0 & 0 & 1 & x_{45} & x_{46} \\
0 & 0 & 0 & 0 & 1 & 0 \\
0 & 0 & 0 & 0 & 0 & 1 \\
 \end{bmatrix}\right\}
\end{eqnarray*}
where $x_{26}, x_{12}, x_{16}, x_{56}, x_{45}, x_{46}\in \C$.  This gives: 
\begin{eqnarray*}
\mathcal{D}_w^1 &=& \mathcal{D}_w^2 = \{(\mathbf{e}_3 \mid \mathbf{e}_6 \mid \mathbf{e}_2 \mid \mathbf{e}_1 \mid \mathbf{e}_5 \mid \mathbf{e}_4)\}\\ 
\mathcal{D}_w^3 &=& \{(\mathbf{e}_3 \mid \mathbf{e}_6 +x_{26}\mathbf{e}_2 \mid \mathbf{e}_2 \mid \mathbf{e}_1 \mid \mathbf{e}_5 \mid \mathbf{e}_4) \} \\
\mathcal{D}_w^4 &=& \{(\mathbf{e}_3 \mid \mathbf{e}_6+x_{16}\mathbf{e}_1+x_{26}(\mathbf{e}_2+x_{12}\mathbf{e}_1) \mid \mathbf{e}_2+x_{12}\mathbf{e}_1 \mid \mathbf{e}_1 \mid \mathbf{e}_5 \mid \mathbf{e}_4)\} \\
\mathcal{D}_w^5 &=& \{(\mathbf{e}_3 + x_{56}\mathbf{e}_2 \mid \mathbf{e}_6+x_{56}\mathbf{e}_5 +x_{16}\mathbf{e}_1+x_{26}(\mathbf{e}_2+x_{12}\mathbf{e}_1) \mid \mathbf{e}_2+x_{12}\mathbf{e}_1 \mid \mathbf{e}_1 \mid \mathbf{e}_5 \mid \mathbf{e}_4)\},\\
\mathcal{D}_w^6 &=& \{(\mathbf{e}_3 + x_{46}\mathbf{e}_1+x_{56}(\mathbf{e}_2+x_{45}\mathbf{e}_1) \mid \mathbf{e}_6 + x_{46}\mathbf{e}_4+x_{56}(\mathbf{e}_5+x_{45}\mathbf{e}_4) + x_{16}\mathbf{e}_1  \cdots\\ 
&& \quad\quad\quad\quad\quad \cdots +x_{26}(\mathbf{e}_2+x_{45}\mathbf{e}_1+x_{12}\mathbf{e}_1) \mid  \mathbf{e}_2+x_{45}\mathbf{e}_1+x_{12}\mathbf{e}_1\mid \mathbf{e}_1 \mid \mathbf{e}_5+x_{45}\mathbf{e}_4 \mid \mathbf{e}_4)\}.
\end{eqnarray*}
\end{example}

\

The set $\mathcal{D}_w$ has an inductive structure.  Indeed, if $w(n)=i$ write $w=vy$ where $y\in S_{n-1}$ and $v=s_{i}s_{i+1}\cdots s_{n-2}s_{n-1}$ as in~\eqref{eqn: factorization}.  Let $V_\bullet = g_ng_{n-1}\cdots g_2wE_{\bullet}\in \mathcal{D}_w$.  Statement (1) of Proposition~\ref{lemma: B_k for  kneq n} implies $g_k':=v^{-1}g_k v\in B_k(y)$ for all $2\leq k \leq n-1$, so $V_\bullet' := g_{n-1}'\cdots g_2' y(E_{\bullet}') \in \mathcal{D}_y\subset \Flags(\C^{n-1})$.  Here $\mathcal{D}_y$ is the collection of flags determined as in Definition~\ref{def.Dset} for $y\in S_{n-1}$ with $R(y)\in \RS(\lambda')$.
We obtain a surjective map:
\begin{eqnarray}\label{eqn: induction Dw set}
B_n(w) \times \mathcal{D}_y \to \mathcal{D}_w ; \; (g_n, V_\bullet')\mapsto V_\bullet.
\end{eqnarray}
The next result gives our affine paving of $\B^{X_\lambda}$.

\begin{theorem} \label{thm: main thm1} Let $w\in S_n$ and $\lambda$ a composition of $n$ such that $R(w)\in \RS(\lambda)$.  There exists an isomorphism of varieties $\widetilde{\varphi}_w: \C^{d_w} \to C_w\cap \B^{X_\lambda}$ with image equal to $\mathcal{D}_w$.  
\end{theorem}
\begin{proof} We argue using induction on $n\geq 1$.  The base case of $n=1$ is trivial since $\Flags(\C) = \{E_\bullet\}$ where $E_\bullet= \Span\{\mathbf{e}_1\}$ in this case.  The required isomorphism $\widetilde{\varphi}_e$ is the map $\mathbf{0} \mapsto E_{\bullet}$. 

Now suppose $n\geq 2$. Let $w=vy$ be the factorization of $w$ from~\eqref{eqn: factorization}, $\lambda'$ the composition of $n-1$ obtained from $\lambda$ by deleting the box labeled by $n$ in $R(w)$, and $X_{\lambda'}$ the restriction of $v^{-1}X_\lambda v$ to $\C^{n-1} = \Span\{ \mathbf{e}_1, \mathbf{e}_2, \ldots, \mathbf{e}_{n-1} \}$.  Consider $C_y\cap \B^{X_{\lambda'}} \subseteq \Flags(\C^{n-1})$.  By induction, there exists an isomorphism of varieties $\widetilde{\varphi}_y: \C^{d_y} \to C_y\cap \B^{X_{\lambda'}}$ with image equal to $\mathcal{D}_y$. Here $d_y = |\inv_{\lambda'}(y)|$ and $\mathcal{D}_y$ is the collection of flags determined as in Definition~\ref{def.Dset} for $y\in S_{n-1}$ with $R(y)\in \RS(\lambda')$.  By Lemma~\ref{lemma: Springer case induction} we have $d_w = d_n + d_y$ and Proposition~\ref{lemma: B_k for  kneq n}(2) gives us an isomorphism $\C^{d_w} \simeq B_n(w)\times \C^{d_y}$.  Using this identification we define 
\[
\widetilde{\varphi}_w:  B_n(w)\times \C^{d_y} \to C_w\cap \B^{X_\lambda},\; \widetilde{\varphi}_w(g_n, \mathbf{x}) = \varphi_w(g_n, \widetilde{\varphi}_y(\mathbf{x}))
\]
where $\varphi_w$ is the isomorphism of Proposition~\ref{prop.isom}.
This is an isomorphism of varieties since both $\varphi_w$ and $\widetilde{\varphi}_y$ are isomorphisms. Using the description of the image of $\widetilde{\varphi}_y$ and definition of $\varphi_w$, it follows that the image of $\widetilde{\varphi}_w$ is equal to the image of the map~\eqref{eqn: induction Dw set}, which is $\mathcal{D}_w$.
\end{proof}

By the theorem, $C_w\cap \B^{X_\lambda} = \mathcal{D}_w$ for all $R(w)\in \RS(\lambda)$.
The main theorem of this section analyzes the structure of the intersections $\mathcal{D}_w\cap \Hess(X_\lambda, h)$ in greater detail.  
We begin with a technical statement generalizing Lemma~\ref{lemma: commutator step1}; it shows that the flags in $\mathcal{D}_w$ satisfy strong linear conditions. 

\begin{lemma}\label{lemma: commutator step2}  Let $R(w)\in \RS(\lambda)$ and $g_k\in B_k(w)$ for each $k$. Furthermore, for each $2\leq k \leq n$ let
\[
\mathbf{v}_j^{(k)} = g_kg_{k-1}\cdots g_2\mathbf{e}_{w(j)} \textup{ for all } 1\leq j \leq n.
\] 
Given $\ell\in [n]$, let $r_\ell$ denote the label of the box directly to the right of the box labeled by $\ell$ in $R(w)$, if such a box exists.  Then 
\begin{eqnarray}\label{eqn: commutator}
(g_kX_\lambda - X_\lambda g_k) \mathbf{v}_j^{(k-1)} = \left\{ \begin{array}{ll} x_{w(k)w(\ell)}\mathbf{v}_k^{(k-1)} & \textup{ if } j=r_{\ell},\, (k, \ell)\in \inv_\lambda^k(w) \\ 0 & \textup{ otherwise}  \end{array}\right.
\end{eqnarray}
where $x_{w(k)w(\ell)}\in \C$ for $(k,\ell)\in \inv_\lambda^k(w)$ is a coordinate of $g_k\in B_k(w)$.
\end{lemma}

\begin{proof} Since $\mathbf{v}_j^{(k-1)} = g_{k-1}\cdots g_2\mathbf{e}_{w(j)}$ and $\mathbf{v}_k^{(k-1)} = g_{k-1}\cdots g_2 \mathbf{e}_{w(k)}$, by Lemma~\ref{lemma: commutator step1} it suffices to show that the matrices $A=g_{k-1}\cdots g_2$ and $g_kX_\lambda - X_\lambda g_k$ commute.

In order to establish this claim, note that Lemma~\ref{lemma: commutator step1} also implies 
\[
\ker(g_kX_\lambda - X_\lambda g_k)=\Span\{\mathbf{e}_{w(j)} \mid j\neq r_\ell  \textup{ for } (k, \ell) \in \inv_\lambda^k(w) \}.
\]  
In particular, we have $\Span\{\mathbf{e}_{w(1)}, \ldots, \mathbf{e}_{w(k)}\}\subseteq \ker(g_kX_\lambda - X_\lambda g_k)$.

If $\mathbf{e}_{w(j)}\notin \ker(g_kX_\lambda - X_\lambda g_k)$ then $j = r_\ell$ for some $(k, \ell)\in \inv_\lambda^k(w)$ and $j \geq k$.  Lemma~\ref{lemma: action of Bk} implies $A\mathbf{e}_{w(j)}= \mathbf{e}_{w(j)}$ and $A\mathbf{e}_{w(k)}= \mathbf{e}_{w(k)}$ so:
\[
(g_kX_\lambda - X_\lambda g_k)A\mathbf{e}_{w(j)} = x_{w_{k}w_{\ell}} \mathbf{e}_{w(k)} = A(g_kX_\lambda - X_\lambda g_k)\mathbf{e}_{w(j)}
\]
by Lemma~\ref{lemma: commutator step1}.  Now suppose $\mathbf{e}_{w(j)}\in \ker(g_kX_\lambda - X_\lambda g_k)$. Lemma~\ref{lemma: action of Bk} implies that for all $j\leq k$, $A\mathbf{e}_{w(j)}\in \Span\{\mathbf{e}_{w(1)}, \ldots, \mathbf{e}_{w(k)}\}$ and if $j\geq k$, $A\mathbf{e}_{w(j)}=\mathbf{e}_{w(j)}$.  Thus $A\mathbf{e}_{w(j)}\in  \ker(g_kX_\lambda - X_\lambda g_k)$ so 
\[
(g_kX_\lambda - X_\lambda g_k)A\mathbf{e}_{w(j)} = \mathbf{0} = A(g_kX_\lambda - X_\lambda g_k)\mathbf{e}_{w(j)}
\]
in this case.
\end{proof}

The following proposition records a key property satisfied by flags in $\mathcal{D}_w$.  This statement is the technical heart of the proof of our main theorem below.

\begin{proposition}\label{prop: difference formula} Suppose $R(w)\in \RS(\lambda)$ and $V_\bullet=(\mathbf{v}_1\mid \mathbf{v}_2\mid \cdots \mid \mathbf{v}_n) \in \mathcal{D}_w$ so $V_{\bullet}= g_ng_{n-1}\cdots g_2w(E_\bullet)$ for $g_k\in B_k(w),\,2\leq k \leq n$.  Let $\ell\in [n]$.  Suppose $\ell$ does not label a box at the end of a row in $R(w)$ and let $r=r_\ell$ denote the label of the box directly to the right of $\ell$ in $R(w)$.  Then
\begin{eqnarray} \label{eqn: difference eqn1}
\mathbf{v}_\ell =   X_\lambda \mathbf{v}_r + \sum_{\substack{\ell < t \leq n \\ (t,\ell)\in \inv_\lambda(w)}} x_{w(t)w(\ell)} \mathbf{v}_t.
\end{eqnarray}
We assume the sum appearing above is zero whenever the index set is empty and $x_{w(t)w(\ell)}\in \C$ for $(t, \ell)\in \inv_\lambda(w)$ are coordinates of $g_t\in B_t(w)$ for each $t$.
\end{proposition}

\begin{proof}  Let $V_\bullet^1 := w(E_\bullet)$ and 
\[
V_\bullet^k := g_k g_{k-1}\cdots g_2w(E_\bullet) = (\mathbf{v}_1^{(k)} \mid \mathbf{v}_{2}^{(k)} \mid \cdots \mid \mathbf{v}_n^{(k)})
\]
with $\mathbf{v}_j^{(k)}$ as in the statement of Lemma~\ref{lemma: commutator step2}.  By definition, $V_\bullet= V_\bullet^n$ and $\mathbf{v}_k = \mathbf{v}_k^{(n)}$ for all $k\in [n]$.  Note that Lemma~\ref{lemma: action of Bk} implies $\mathbf{v}_k^{(k-1)}= \mathbf{e}_{w(k)} =  \mathbf{v}_k^{(k)}$ for all $k\in [n]$.  

Since $r$ appears in the box directly to the right of $\ell$ in $R(w)$, we know $X_\lambda \mathbf{e}_{w(r)}= \mathbf{e}_{w(\ell)}$. We will show 
\begin{eqnarray} \label{eqn: difference induction}
\mathbf{v}_\ell^{(k)} - X_\lambda \mathbf{v}_r^{(k)} = \sum_{\substack{ \ell< t\leq k\\ (t,\ell)\in \inv_\lambda(w)}} x_{w(t)w(\ell)} \mathbf{v}_t^{(k)}
\end{eqnarray}
for every $2\leq k\leq n$.  This gives us~\eqref{eqn: difference eqn1} when $k=n$, proving the desired result.

We proceed by induction on $k$.  When $k=2$ the formula in Lemma~\ref{lemma: commutator step1} implies 
\[
(g_2X_\lambda -X_\lambda g_2)\mathbf{e}_{w(r)} = \sum_{(2, \ell)\in\inv_\lambda (w)} x_{w(2)w(\ell)}\mathbf{e}_{w(2)}
\]
where the RHS is zero whenever the index set is empty, that is, whenever $(2,\ell)\notin \inv_\lambda^2(w)$.  Applying the identities $\mathbf{v}_\ell^{(2)}= g_2\mathbf{e}_{w(\ell)}$, $\mathbf{v}_r^{(2)}=g_2\mathbf{e}_{w(r)}$, and $\mathbf{e}_{w(2)} = \mathbf{v}_2^{(2)}$ we obtain
\[
\mathbf{v}_{\ell}^{(2)}-X_\lambda \mathbf{v}_{r}^{(2)} =  \sum_{(2, \ell)\in\inv_\lambda (w)} x_{w(2)w(\ell)}\mathbf{v}_2^{(2)}.
\]
as desired. 

Next, assume $k>2$ and~\eqref{eqn: difference induction} holds for $k-1$. Applying $g_k$ to both sides of this equation, we obtain
\begin{equation}\label{eqn: induction step}
\mathbf{v}_\ell^{(k)} - g_kX_\lambda \mathbf{v}_r^{(k-1)} = \sum_{\substack{\ell< t\leq k-1\\ (t, \ell)\in \inv_\lambda(w)}} x_{w(t)w(\ell)}\mathbf{v}_t^{(k)}.
\end{equation}
The formula of Lemma~\ref{lemma: commutator step2} shows that
\[
g_kX_\lambda \mathbf{v}_r^{(k-1)} = \left\{ \begin{array}{ll} X_\lambda g_k\mathbf{v}_r^{(k-1)} + x_{w(k)w(\ell)}\mathbf{v}_k^{(k-1)} & \textup{ if } (k, \ell)\in \inv_\lambda^k(w) \\ X_\lambda g_k\mathbf{v}_r^{(k-1)} & \textup{ otherwise.}  \end{array}\right.
\]
Substitute this formula for $g_kX_\lambda \mathbf{v}_r^{(k-1)}$ into the LHS of~\eqref{eqn: induction step}.  Since $ g_k\mathbf{v}_r^{(k-1)} = \mathbf{v}_r^{(k)}$ and $\mathbf{v}_k^{(k-1)} = \mathbf{v}_k^{(k)}$, this substitution yields the desired result.
\end{proof}

We can now prove our main result.

\begin{theorem}\label{thm: main thm2} Let $h: [n]\to [n]$ be a Hessenberg function such that $h(i)<i$ for all $i$, $\lambda$ be a composition of $n$, and $w\in S_n$ such that $R(w)\in \RS_h(\lambda)$. Suppose $V_\bullet= g_ng_{n-1}\cdots g_2wE_\bullet\in \mathcal{D}_w$ where $g_k\in B_k(w)$ for all $2\leq k\leq n$.  
\begin{enumerate}
\item  We have $V_\bullet \in \Hess(X_\lambda, h)$ if and only if $x_{w(k) w(\ell)}=0$ for all $(k,\ell)\in \inv_\lambda(w)\setminus \inv_{\lambda, h}(w)$, where $x_{w(k)w(\ell)}\in \C$ is a coordinate of $g_k$.
\item The isomorphism of Theorem~\ref{thm: main thm1} restricts to an isomorphism $\C^{d_{w,h}} \rightarrow C_w\cap \Hess(X_\lambda ,h )$ where $d_{w,h} = |\inv_{\lambda, h}(w)|$. 
\end{enumerate}
In particular, there is an affine paving of $\Hess(X_\lambda,h)$ with affine cells obtained by intersecting $\Hess(X_\lambda,h)$ with the Schubert cells.
\end{theorem}
\begin{proof}  By Lemma~\ref{lemma: action of Bk}, each $g_k\in B_k(w)$ preserves the kernel of $X_\lambda$.  If follows immediately that if $r$ labels a box in the first column of $R(w)$ (that is, if $\mathbf{e}_{w(r)}\in \ker(X_\lambda)$), then $\mathbf{v}_r = g_n g_{n-1}\cdots g_2\mathbf{e}_{w(r)}\in \ker(X_\lambda)$.  Thus, in this case, the condition that  $X_\lambda \mathbf{v}_r \in \Span\{\mathbf{v}_1, \mathbf{v}_2, \ldots, \mathbf{v}_{h(r)}\}$ is vacuously true.

Next, we consider the case in which $r$ is not in the first column of $R(w)$. Rewriting formula~\eqref{eqn: difference eqn1} from Proposition~\ref{prop: difference formula} we have,
\[
X_\lambda \mathbf{v}_r = \mathbf{v}_\ell- \sum_{\substack{\ell<t\leq n\\ (t,\ell)\in \inv_\lambda(w)}} x_{w(t)w(\ell)}\mathbf{v}_t
\]
where $\ell$ is the label of the box immediately to the left of $r$ in $R(w)$.  Since $(t,\ell)\in \inv_\lambda (w)\setminus \inv_{\lambda, h}(w)$ if and only if $t>\ell$ and $h(r)<t< r$ it follows immediately that $X_\lambda \mathbf{v}_r \in \Span\{\mathbf{v}_1, \mathbf{v}_2, \ldots, \mathbf{v}_{h(r)}\}$ if and only if $x_{w(t)w(\ell)}=0$ for all $(t,\ell)\in \inv_\lambda (w)\setminus \inv_{\lambda, h}(w)$. This proves statement (1).

Let $\mathbf{x} = (x_{w(k)w(\ell)})_{(k,\ell) \in \inv_\lambda(w)}$.  By definition, $\widetilde{\varphi}_w(\mathbf{x}) = g_ng_{n-1}\ldots g_2 wE_\bullet$ where $g_k$ has coordinates $(x_{w(k)w(\ell_1)} , \ldots, x_{w(k), w(\ell_d)})$ for $d=d_k$ and $\inv_\lambda^k(w) = \{ (k, \ell_1), \ldots, (k, \ell_d) \}$.  Statement (1) implies $\widetilde{\varphi}_w(\mathbf{x}) \in C_w\cap \Hess(X_\lambda, h) = \mathcal{D}_w\cap \Hess(X_\lambda, h)$ if and only if $x_{(w(k)w(\ell))} = 0$ for all $(k,\ell) \notin \inv_{\lambda, h}(w)$.  Thus, the restriction of $\widetilde{\varphi}_w$ obtained by setting $x_{(w(k)w(\ell))}=0$ for all $(k,\ell) \notin \inv_{\lambda, h}(w)$ yields an isomorphism $\C^{d_{w,h}} \to C_w\cap \Hess(X_\lambda, h)$.  This proves (2) and the final assertion of the theorem follows from Remark~\ref{rem.paveit}.
\end{proof}

\begin{remark} One can also recover the results of Theorem~\ref{thm: main thm2} using similar methods as the second author in \cite{Precup2013}. It is an exercise to show that the formula for $\dim(C_w\cap \Hess(X_\lambda, h))$ given in Proposition 3.7 of \cite{Precup2013} is equal to $|\inv_{\lambda,h}(w)|$.  
\end{remark}

Our next example continues the work of Examples~\ref{ex: induction} and~\ref{ex: Dw set}.

\begin{example}  
Let $n=6$ and $\lambda = (2,2,2)$ and $w=[3,6,2,1,5,4]$.   
Consider the Hessenberg function $h=(0, 0, 1, 1, 3, 4)$.  Then 
\begin{eqnarray*}
\mathcal{D}_w\cap\Hess(X_\lambda, h) &=&  \{(\mathbf{e}_3 + x_{46}\mathbf{e}_1+x_{56}(\mathbf{e}_2+x_{45}\mathbf{e}_1) \mid \mathbf{e}_6 + x_{46}\mathbf{e}_4+x_{56}(\mathbf{e}_5+x_{45}\mathbf{e}_4) + x_{16}\mathbf{e}_1  \cdots\\ 
&& \quad\quad\quad \cdots +x_{26}(\mathbf{e}_2+x_{45}\mathbf{e}_1) \mid  \mathbf{e}_2+x_{45}\mathbf{e}_1 \mid \mathbf{e}_1 \mid \mathbf{e}_5+x_{45}\mathbf{e}_4 \mid \mathbf{e}_4)\}.
\end{eqnarray*}
and 
\begin{eqnarray*}
\mathcal{D}_w &=&  \{(\mathbf{e}_3 + x_{46}\mathbf{e}_1+x_{56}(\mathbf{e}_2+x_{45}\mathbf{e}_1) \mid \mathbf{e}_6 + x_{46}\mathbf{e}_4+x_{56}(\mathbf{e}_5+x_{45}\mathbf{e}_4) + x_{16}\mathbf{e}_1  \cdots\\ 
&& \quad\quad\quad\quad \cdots +x_{26}(\mathbf{e}_2+x_{45}\mathbf{e}_1+x_{12}\mathbf{e}_1) \mid  \mathbf{e}_2+x_{45}\mathbf{e}_1+x_{12}\mathbf{e}_1\mid \mathbf{e}_1 \mid \mathbf{e}_5+x_{45}\mathbf{e}_4 \mid \mathbf{e}_4)\}.
\end{eqnarray*}
Since $(4, 3)\in \inv_\lambda(w)\setminus \inv_{\lambda,h}(w)$ and $(w(4),w(3))= (1,2)$, to get the former paving from the latter we set $x_{12}=0$.
\end{example}


\section{Geometric and Combinatorial Properties} \label{sec.applications}

In this section, we use the affine paving from above to study the geometry of $\Hess(X_\lambda,h)$ using the combinatorics of $h$-strict tableaux.  We prove two main results, generalizing known facts about the Springer fiber to the Hessenberg varieties studied in this paper.  The first is that, when $\lambda$ is a partition and the cell $C_w\cap \Hess(X_\lambda, h)$ has maximal dimension, then $R(w)\in \RS_h(\lambda)$ is a standard tableau.  Our second result proves that $\Hess(X_\lambda,h)$ is connected.  

Assume for now that $\lambda$ is a partition.  If $R$ is a row-strict tableau of shape $\lambda$, we let $\std(R)$ denote the standard tableau obtained by reordering the entries in each column so that they increase from top to bottom.  The next result generalizes Theorem 3.5 in \cite{Precup-Tymoczko2019}. 

\begin{lemma} \label{std} Let $\lambda$ be a partition of $n$.
If $R\in \RS_h(\lambda)$ then $\std(R)\in \RS_h(\lambda)$.  
\end{lemma}
\begin{proof}
Suppose $a_i$ is the entry in row $i$ and column $k>1$ of $\std(R)$.  Then $a_i$ is greater than precisely $i-1$ other entries of the $k$-th column in $R$. Since $R\in \RS_h(\lambda)$, $h(a_i)$ is greater than or equal to at least $i$ distinct entries in column $k-1$.  Thus $h(a_i)$ is greater than or equal to the entry in the box to the immediate left of $a_i$ in $\std(R)$. 
This implies that $\std(R)\in\RS_h(\lambda)$.
\end{proof}

We now prove the first of the main results in this section.  In general, the varieties $\Hess(X_\lambda,h)$ need not be equidimensional, but the following theorem shows that the maximal dimension cells in our affine paving correspond to standard tableaux.  First we require the following definition.

\begin{definition}
Suppose $R\in \RS_h(\lambda)$ has $m$ columns.  Let $d_R(i, j)$ for $1\le i\le j\le m$ be the number of Hessenberg inversions $(k, \ell)$ with $k$ in column $i$ and $\ell$ in column $j$ of $R$. 
\end{definition}
Note that in the definition above, $i$ and $j$ may be equal.  In the proof of the theorem below, we will consider a process that alters the entires of a tableau $R$.  We use the notation $d(i, j)$ to refer to the number of Hessenberg inversions in columns $i$ and $j$ at each step of the process in our proof.

\begin{theorem}\label{thm.maximalcells} If $w\in S_n$ such that $\dim(C_w\cap \Hess(X_\lambda,h)) = |\inv_{\lambda, h}(w)|$ is maximal, then $R(w)$ is a standard tableau.  
\end{theorem} 

\begin{proof} Suppose $R\in \RS_h(\lambda)$ is not a standard tableau.  Write $R=R(y)$ and $\std(R)=R(w)$ for some $y,w\in S_n$.  We claim that $d_R(i, j) \leq d_{\std(R)}(i, j)$ for all pairs $(i, j)$, and furthermore, at least one inequality is strict.  This implies $|\inv_{\lambda, h}(y)|< |\inv_{\lambda,h}(w)|$, proving the theorem.   

We prove this claim by describing a process which, when applied to any three columns $i, j, j+1$, will sort them.  Furthermore, each step in this process will not increase the value of $d(i, j)$. (If $i=j$, then there are only two columns.)  This will imply that $d_R(i, j) \leq d_{\std(R)}(i, j)$.     

We now describe this process; the idea is to `bubble sort columns'.  If there is an inversion between two adjacent boxes in column $i$, we exchange them as well as the entries in that same row for column $j$ and column $j+1$.  We consider all rows to have the same number of squares and count blank squares as having a value of $+\infty$.  Note that this means that, in the middle of this process, we may end up with a composition of $n$ that is not a partition.  Continue this until column $i$ is increasing.  Then do this for column $j$; bubble sort out the inversions while simultaneously exchanging the corresponding entries in column $j+1$.  Finally, bubble sort column $j+1$.  This results in the entries of columns $i$, $j$, and $j+1$ being reordered so that they increase from top to bottom (i.e., we obtain the corresponding columns of $\std(R)$).

First we analyze what happens when $i=j$.  Our goal is to prove that this process does not increase the number of Hessenberg inversions $(k,\ell)$ with $k$ and $\ell$ in column $i$.  In this case, any such pair $(k, \ell)$ will still be a pair after sorting column $i$ because $k>\ell$ so $k$ will still label a box below $\ell$ in column $i$, and the number in the box to the right of $\ell$ does not change.  Note that sorting column $i$ does not change the fact that columns $i$ and $i+1$ are $h$-strict.

Now consider sorting column $j+1=i+1$.  Suppose $\ell>s$ are adjacent entries in column $i$ (so $\ell$ appears directly below $s$) and that we swap $r_\ell$ and $r_s$ (the entries to the right of $\ell$ and $s$, respectively) as we sort column $i+1$, so $r_\ell<r_s$.  We first prove that the $h$-strictness of columns $i$ and $i+1$ is preserved.  Indeed, if $s\leq h(r_s)$ and $\ell\leq h(r_\ell)$ then 
\[
\ell\leq h(r_\ell) \leq h(r_s)  \textup{ and } s<\ell \leq h(r_\ell)
\]
since $r_\ell<r_s$ implies $h(r_\ell)\leq h(r_s)$.  Now, if we lose any pair counted by $d(i, j)$, it must be of the form $(k, s)$, where $k>s$ appears below $s$ in column $i$ and $h(r_\ell)< k\leq h(r_s)$. Note that $k\neq\ell$ in this case, since $\ell\leq h(r_s)$. Since $k>s$ and $\ell$ appears immediately below $s$, we get that $k>\ell$ labels a box below $\ell$ in column $i$.  Since $k$ is below $\ell$ and $r_s$ now labels the box directly to the right of $\ell$ after the swap, we gain the Hessenberg inversion $(k, \ell)$.  Thus $d(i, j)$ does not decrease.

Now suppose $i < j$ and consider the operations on column $i$ first.  Because every move swaps entire rows, bubble sorting column $i$ will not result in the loss of any Hessenberg inversions, and the columns $j$ and $j+1$ remain $h$-strict.  Now consider sorting column $j$.  This clearly does not affect the inversions in $d(i,j)$ for the same reason, namely the columns $j$ and $j+1$ are being simultaneously swapped.  Columns $j$ and $j+1$ remain $h$-strict in this case also.   

Finally, consider the operations on column $j+1$.  Say $s$ is directly above $\ell$ in column $j$ (so $s<\ell$) and $r_s>r_\ell$, as above. Swapping $r_s$ and $r_\ell$ preserves the $h$-strictness of columns $j$ and $j+1$; the argument is the same as above. If we lose any Hessenberg inversion $(k,s)$ where $k$ appears in column $i$, then we must have $h(r_\ell)<k\leq h(r_s)$, just as above. But we gain the pair $(k,\ell)$ after swapping $r_s$ and $r_\ell$ since $\ell \leq h(r_\ell)<k$ and $k\leq h(r_s)$.  Once again, we see that $d(i,j)$ does not decrease.

Now we show that $d_{\std(R)}(i, j)>d_R(i, j)$ for some $(i, j)$.  Let column $k$ be the last column of $R$ whose entries are not already increasing from top to bottom.  Consider the first instance the bubble sorting process swaps two elements when $i=j=k$.  Say we swap $\ell$ and $s$ with $\ell>s$.  Since column $j+1$ begins sorted, we have $r_\ell<r_s$.  Now $\ell \le h(r_\ell)$, so we gain the Hessenberg inversion $(\ell, s)$.  Moreover, as we showed in the general case where $i=j$, we do not lose any pairs counted by $d(i,i)$ when sorting column $i$, and every pair we lose when sorting column $i+1$ matches with one we gain upon doing so.  Thus $d_{\std(R)}(k, k)>d_R(k, k)$, as desired.
\end{proof}

We demonstrate the algorithm in the previous proof with an example.

\begin{example}
Let $n=12$ and $h: [12]\to [12]$ the Hessenberg function defined by $h(i) = \max\{0, i-2\}$.  Consider beginning with the following $h$-strict tableaux $R$ and $\std(R)$:
\[
R =\begin{ytableau}
2&4&8&10\\
1&5&7&11\\
3&9&12\\
6
\end{ytableau}\quad\longrightarrow\quad
\std(R) = \begin{ytableau}
1&4&7&10\\
2&5&8&11\\
3&9&12\\
6
\end{ytableau}
\]
When $i=j=1$, we check that $d_R(i,j)$ does not decrease during the operations described in the proof. Each step of the process is displayed below; in this case there are only two steps.
\[
\begin{ytableau}
2&4\\
1&5\\
3&9\\
6
\end{ytableau}\quad\longrightarrow\quad
\begin{ytableau}
1&5\\
2&4\\
3&9\\
6
\end{ytableau}\quad\longrightarrow\quad
\begin{ytableau}
1&4\\
2&5\\
3&9\\
6
\end{ytableau}
\]
We see $d_R(1,1)=2\leq 3=d_{\std(R)}(1,1)$, as desired.  Next, when $i=1$ and $j=2$:
\[
\begin{ytableau}
2&4&8\\
1&5&7\\
3&9&12\\
6
\end{ytableau}\quad\longrightarrow\quad
\begin{ytableau}
1&5&7\\
2&4&8\\
3&9&12\\
6
\end{ytableau}\quad\longrightarrow\quad
\begin{ytableau}
1&4&8\\
2&5&7\\
3&9&12\\
6
\end{ytableau}\quad\longrightarrow\quad
\begin{ytableau}
1&4&7\\
2&5&8\\
3&9&12\\
6
\end{ytableau}
\]
and we again have $d_R(1,2)=1\leq 1 = d_{\std(R)}(1,2)$.  Since column 3 is the last column of $R$ in which the values do not increase from top to bottom, we check that when $i=j=3$, $d_R(i,j)<d_{\std(R)}(i,j)$. 
\[
\begin{ytableau}
8&10\\
7&11\\
12\\
\end{ytableau}\quad\longrightarrow\quad
\begin{ytableau}
7&10\\
8&11\\
12\\
\end{ytableau}
\]
 Indeed, $d_R(3,3)=0$ and $d_{\std}(3,3)=1$.
\end{example}

We now return to the setting in which $\lambda$ is any composition of $n$.  We prove that each Hessenberg variety $\Hess(X_\lambda, h)$ is connected. Recall that the dimension of the $0$-cohomology group of any algebraic variety is equal to the number of connected components of that space.  Thus, by Lemma~\ref{betti} and Theorem~\ref{thm: main thm2} it suffices to show that there is a unique permutation $w\in S_n$ such that $\dim(C_w\cap \Hess(X_\lambda, h))=0$.

\begin{theorem}\label{thm.connected} For any Hessenberg function $h\in \mathcal{H}$ and composition $\lambda$, the Hessenberg variety $\Hess(X_\lambda, h)$ is connected whenever it is nonempty.
\end{theorem} 

\begin{proof}
By Remark~\ref{rem.conjugation}, we may assume without loss of generality that $\lambda$ is a partition of $n$.  It suffices to prove that there exists a unique $h$-strict tableau that has no Hessenberg dimension pairs. Consider the following algorithm.
Begin at the right-most column of $\lambda$.  Label the boxes in this column from top to bottom, assigning $n$ to the first box and decreasing by one at a time.  Then move to the next right-most column, filling the boxes from top to bottom with the largest available number subject to the constraint that doing so does not violate $h$-strictness.  Continue in this way and denote the resulting tableau by $R_0$.  We first show that this algorithm results in an $h$-strict tableau whenever $\Hess(X_\lambda, h)\neq \emptyset$; then we will show that $R_0$ is the unique tableaux satisfying the desired properties.

If $\Hess(X_\lambda, h)$ is nonempty, then at least one intersection $C_w\cap \Hess(X_\lambda, h)$ must be nonempty.  Let $R=R(w)$ be the corresponding $h$-strict tableau.  We show that $R$ can be transformed into an $h$-strict tableau obtained via the algorithm above by a sequence of swaps.  Given $\ell\in [n]$, recall that we denote by $r_\ell$ the label of the box directly to the right of $\ell$, if it exists.

Going in order from top to bottom in each column, starting the right-most column and moving left, take the first square in $R$ whose label differs from that dictated by the algorithm. 
In other words, we consider the first square in $R$ with label $\ell$ such that there exists $k>\ell$ with $k\leq r_\ell$ and such that $k$ appears below the box containing $\ell$ and in the same column or in any column strictly to the left of the column containing $\ell$.  We may assume $k$ is the maximal with respect to these properties.  Let $\ell_m$ and $k_m$ be the values $m$ boxes to the left of $\ell$ and $k$ in tableau $R$, respectively.  

Define $R'$ to be the following $h$-strict tableau obtained from $R$.  Exchange the positions of $k$ and $\ell$ in $R$.  If this is an $h$-strict tableau, then stop.  Otherwise, the only problem that can arise is that $k_1 > h(\ell)$.  In that case, we must also have $k_1>\ell_1$. Now exchange the positions of $k_1$ and $\ell_1$; if this results in an $h$-strict tableau then stop.  Otherwise note that swapping $k_1$ into its new row does not change the fact that that row is $h$-strict.  The only problem that can occur is if $k_2 > h(\ell_1)$ (and therefore $k_2>\ell_2$).  Continue this process of swapping until the tableau is $h$-strict.  Indeed, because $k$ originally appeared in a `later' box than $\ell$, we will always be able to perform this swap until the entire tableau is $h$-strict. 
The resulting $R'$ is $h$-strict and has labels as dictated by the algorithm up through the box which now contains $k$.

Repeating the operation of the previous two paragraphs yields an $h$-strict tableau whose entries are dictated by the algorithm.  This shows that the algorithm does indeed produce a completed $h$-strict tableau, $R_0$.
Now consider any pair $(k, \ell)$ with $k>\ell$ such that $k$ labels a box in $R_0$ below $\ell$ and in the same column or in any column to the left of the column containing $\ell$.  By construction, placing $k$ in the box containing $\ell$ would lead to a violation of $h$-strictness, i.e., $k > h(r_\ell)$.  This proves that $R_0$ has no Hessenberg inversions.

Finally, we show there are no other $h$-strict tableau $R(w)$ with $|\inv_{\lambda,h}(w)|=0$.  Suppose $R(w)=R\in \RS_h(\lambda)$ such that $R\neq R_0$ and that in the first box (using the same ordering as above) in which $R$ differs from $R_0$, $R$ is labeled by $\ell$ rather than a $k$.  Then $k>\ell$ since our algorithm for constructing $R_0$ always chooses the largest available number.  The value $k$ must be placed in a `later' box of $R$ than the box containing $\ell$, so $(k, \ell)$ is an inversion of $w$.  Since $k$ could have been placed in the box containing $\ell$ (as it is in $R_0$), we have that $k\le h(r_\ell)$, so $(k, \ell)$ is a Hessenberg inversion pair.  This shows that all other $h$-strict tableaux have at least one Hessenberg inversion pair.  Thus $R_0$ is the unique $h$-strict tableau with the desired properties. 
\end{proof}

In the Springer fiber case, the unique row-strict tableau without any Springer inversions is the base filling from Definition~\ref{def: fixed element} above.  However, in general the base filling may not be $h$-strict, as the next example demonstrates.

\begin{example}
Let $n=12$ and $h:[12]\to [12]$ the Hessenberg function defined by $h(i) = \max\{0, i-3\}$.  The base filling for $\lambda = (4,4,3,1)$ and the $h$-strict tableau $R_0$ constructed in the proof of Theorem~\ref{thm.connected} are displayed below.
\[
R(e) = \begin{ytableau}
4 & 7 & 10 & 12 \\
3 & 6 & 9 & 11 \\
2 & 5 & 8 \\
1
\end{ytableau}
\quad\quad\quad
R_0 = \begin{ytableau}
3 & 6 & 9 & 12 \\
2 & 5 & 8 & 11 \\
4 & 7 & 10 \\
1
\end{ytableau}
\]
Note that $R(e)$ is not $h$-strict, since $10$ and $12$ appear as consecutive entries in the first row, but $10\nleq 9=h(12)$.  The unique row-strict tableau with no Hessenberg inversions is $R_0$.
\end{example}

Finally, we conclude with an example which shows that analogous Hessenberg varieties defined for other classical groups need not be connected.  

\begin{example}\label{ex.TypeC} Consider the symplectic group $SP_{4}(\C)$.  Our convention is that the inner product on $\C^4=\mathrm{span}\{\mathbf{e}_1, \mathbf{e}_2,\mathbf{e}_3, \mathbf{e}_4\}$ is determined by the following matrix
\[
s_J = \left[ \begin{array}{c|c}0 & J \\ \hline -J & 0\end{array}\right] \textup{ where } J=\begin{bmatrix}0 & 1\\ 1 & 0\end{bmatrix}.
\]
In other words, given $g\in GL_4(\C)$ we have $g\in SP_4(\C)$ if and only if $g^{T}s_Jg=s_J$ where $g^{T}$ denotes the transpose of $g$.  The symplectic group has Lie algebra $\mathfrak{sp}_{4}(\C)=\{X\in \mathfrak{gl}_4(\C)\mid s_JX=-X^{T}s_J \}$. 

The flag variety of $SP_4(\C)$ consists of full flags of isotropic subspaces.  Let
\[
X=\begin{bmatrix} 0 & 0 & 1 & 0\\ 0 & 0 & 0 & 1\\ 0 & 0 & 0 & 0\\ 0 & 0 & 0 & 0 \end{bmatrix}\in \mathfrak{sp}_4(\C)
\] 
and  $h$ be the Hessenberg function $(0,1,1,3)$.  This Hessenberg function gives a well-defined subvariety of the flag variety for $SP_4(\C)$ using the Hessenberg space $H=H(h)\cap \mathfrak{sp}_4(\C)$; see, for example, \cite[Section 2.2]{Precup2013} for the definition of the Hessenberg variety in an arbitrary flag variety. The variety of isotropic flags $V_\bullet=(\mathbf{v}_1\mid \mathbf{v}_2\mid\mathbf{v}_3\mid \mathbf{v}_4)$ such that:
\begin{eqnarray}\label{eqn: type C}
X\mathbf{v}_1=\mathbf{0};\;\; X\mathbf{v}_2, X\mathbf{v}_3\in V_1=\Span\{\mathbf{v}_1\};\;\textup{ and }\; X\mathbf{v}_4\in V_3=\Span\{\mathbf{v}_1, \mathbf{v}_2, \mathbf{v}_3\}
\end{eqnarray}
is the Hessenberg variety $\Hess(X,H)$ in the flag variety of $SP_4(\C)$. It is a subvariety of the Springer fiber corresponding to $X$.  In this small dimensional case, it is easy to show that $\Hess(X,H)$ is paved by affines, and this paving is obtained by intersecting $\Hess(X,H)$ with each Schubert cell.  Let $s_1$ be the element of the Weyl group of $SP_4(\C)$ such that $s_1\mathbf{e}_1=\mathbf{e}_2$, $s_1\mathbf{e}_2=\mathbf{e}_1$, $s_1\mathbf{e}_3=\mathbf{e}_4$, $s_1\mathbf{e}_4=\mathbf{e}_3$, and let $e$ denote the identity element.  The Schubert cell $C_{s_1}$ consists of all flags of the form:
\[
(\mathbf{e}_2+c\,\mathbf{e}_1 \mid \mathbf{e}_1\mid \mathbf{e}_4-c\,\mathbf{e}_3\mid \mathbf{e}_3) \textup{ for some } c\in \C
\]
and $C_e=\{E_\bullet\}$.  The conditions from~\eqref{eqn: type C} now imply:
\[
C_e\cap \Hess(X,H) = \{E_\bullet\}\;\; \textup{ and } \;\; C_{s_1}\cap \Hess(X,H) = \{s_1E_\bullet\}
\]
so $\dim(C_e\cap \Hess(X,H))=\dim(C_{s_1}\cap \Hess(X,H))=0$. Since the $0$-cohomology group of $\Hess(X,H)$ has dimension 2, $\Hess(X,H)$ is not connected.

\end{example}


\end{document}